\documentclass[reqno]{amsart}

\usepackage{tabu}
\usepackage{amssymb}
\usepackage{mathtools}
\usepackage{a4wide,amsmath}
\usepackage{mathrsfs}
\usepackage{amsthm}
\numberwithin{equation}{section}
\numberwithin{figure}{section}
\numberwithin{table}{section}
\usepackage{bbm}
\usepackage{subfig}
\usepackage{enumerate}
\usepackage{needspace}
\usepackage[section]{placeins}
\usepackage{graphicx}		  
\usepackage{ifpdf}
\ifpdf
\DeclareGraphicsExtensions{.pdf,.eps,.jpg,.png}	
\usepackage[suffix=]{epstopdf}
\fi
\usepackage{xcolor}
\usepackage[utf8]{inputenc}
\usepackage{hyperref}
\hypersetup{hidelinks}

\long\def\MSC#1\EndMSC{\def\arg{#1}\ifx\arg\empty\relax\else
	{\narrower\noindent%
		{2020 Mathematics Subject Classification}: #1\\} \fi}
\long\def\PACS#1\EndPACS{\def\arg{#1}\ifx\arg\empty\relax\else
	{\narrower\noindent%
		{PACS numbers}: #1}\fi}
\long\def\KEY#1\EndKEY{\def\arg{#1}\ifx\arg\empty\relax\else
	{\narrower\noindent%
		Keywords: #1\\}\fi}

\theoremstyle{plain}
\newtheorem{theorem}{Theorem}[section]
\newtheorem{lemma}[theorem]{Lemma}
\newtheorem{proposition}[theorem]{Proposition}
\newtheorem{corollary}[theorem]{Corollary}
\theoremstyle{definition}
\newtheorem{definition}[theorem]{Definition}

\newtheorem{assumption}[theorem]{Assumption}
\theoremstyle{remark}
\newtheorem{remark}[theorem]{Remark}

\newcommand{\norm}[1]{\lVert#1\rVert}
\newcommand{\abs}[1]{\lvert#1\rvert} 
\newcommand{\inner}[1]{\langle#1\rangle} 
\newcommand{\redel}{\mathop{\textup{Re}}}
\newcommand{\imdel}{\mathop{\textup{Im}}}

\newcommand{\suppm}{\mathop{\textup{supp}}}
\newcommand{\R}{\mathbb{R}}
\newcommand{\N}{\mathbb{N}}
\newcommand{\C}{\mathbb{C}}

\newcommand{\I}{\mathrm{i}}    
\newcommand{\di}{\mathrm{d}}   

\newcommand{\upR}{^{\textup{R}}}
\newcommand{\upI}{^{\textup{I}}}
\newcommand{\uptR}{^{\!\textup{R}}}
\newcommand{\uptI}{^{\!\textup{I}}}

\newcommand{\LambdaR}{\Lambda\uptR}
\newcommand{\LambdaI}{\Lambda\uptI}
\newcommand{\AR}{A\upR}
\newcommand{\AI}{A\upI}
\newcommand{\Hsa}{\mathcal{H}_{\textup{SA}}}
\newcommand{\DLambda}{\textup{D}\mkern-1.5mu \Lambda}
\newcommand{\DLambdah}{\textup{D}\mkern-1.5mu \widehat{\Lambda}}
\newcommand{\DPhi}{\mathop{\textup{D}\Phi}\nolimits}

\begin{document}
	\title[Reconstruction of anisotropic self-adjoint inclusions]{Direct reconstruction of anisotropic self-adjoint inclusions in the Calder\'on problem}
	
	\author[H.~Garde]{Henrik~Garde}
	\address[H.~Garde]{Department of Mathematics, Aarhus University, Aarhus, Denmark.}
	\email{garde@math.au.dk}
	
	\author[D.~Johansson]{David~Johansson}
	\address[D.~Johansson]{Department of Mathematics, Aarhus University, Aarhus, Denmark.}
	\email{johansson@math.au.dk}
	
	\author[T.~Zacharopoulos]{Thanasis~Zacharopoulos}
	\address[T.~Zacharopoulos]{Department of Mathematics, Aarhus University, Aarhus, Denmark.}
	\email{thanzacharop@math.au.dk}
	
	\begin{abstract}
		We extend the monotonicity method for direct exact reconstruction of inclusions in the partial data Calder\'on problem, to the case of general anisotropic conductivities in any spatial dimension $d\geq 2$. From a local Neumann-to-Dirichlet map, we give reconstruction methods of inclusions based on unknown anisotropic self-adjoint perturbations to a known anisotropic conductivity coefficient. This additionally provides new insights into the non-uniqueness issues of the anisotropic Calder\'on problem.
		
		The main assumption is a definiteness condition for the perturbations near the outer inclusion boundaries. Beyond this condition, they are $L^\infty$-perturbations that may be indefinite away from the outer inclusion boundaries, and with no boundary regularity requirement for the inclusions. Alternatively, we allow extreme parts that are perfectly insulating or perfectly conducting, in which case we require Lipschitz regularity of the outer inclusion boundaries. 
	\end{abstract}	
	\maketitle
	
	\KEY
	anisotropic Calder\'on problem, inclusion detection, monotonicity method.
	\EndKEY
	
	\MSC
	35R30, 35R05, 47H05.
	\EndMSC
	
	\tableofcontents
	\clearpage
	
	\section{Introduction} \label{sec:intro}
	
	The anisotropic Calder\'on problem, the inverse problem of determining an elliptic matrix-valued conductivity coefficient $A$ from boundary measurements of the conductivity equation
	\begin{equation*}
		-\nabla\cdot(A\nabla u) = 0 \text{ in } \Omega,
	\end{equation*}
	is of major interest in the inverse problems community. Unlike the isotropic (scalar-valued) case, the anisotropic Calder\'on problem suffers from significant non-uniqueness issues. Indeed if boundary measurements are taken on $\Gamma\subseteq\partial\Omega$, for any $H^1$-diffeomorphism $\Phi\colon \overline{\Omega}\to\overline{\Omega}$ for which $\Phi|_{\Gamma}$ is the identity, the conductivity coefficient defined via the push-forward
	\begin{equation} \label{eq:pushforward}
		\widetilde{A} = \frac{\DPhi A \DPhi^{\textup{T}}}{\abs{\det\DPhi}}\circ \Phi^{-1}
	\end{equation}
	produces the same local Neumann-to-Dirichlet (ND) map as the coefficient $A$. For $\Gamma = \partial\Omega$ and with spatial dimension $d=2$, the obstructions to uniqueness are characterized by~\eqref{eq:pushforward} for real-valued symmetric $L^\infty$-conductivities~\cite{Astala2005,Sylvester1990}. In spatial dimension $d\geq 3$, the full extent of the non-uniqueness issues (beyond real-analytic coefficients) remains one of the most important unanswered questions in the inverse problems community~\cite{Lee1989,Uhlmann2009}. In a fractional Calder\'on problem, using the stronger non-local properties, the obstructions to uniqueness are characterized for $d\geq 2$~\cite{Feiz2025}. 
	
	Despite knowledge on the possible obstructions, the non-uniqueness has prevented development of practically useful reconstruction methods. Unique recovery is only guaranteed with very strong restrictions on the considered class of coefficients~\cite{Alessandrini2017,Alessandrini2018,Foschiatti2025}.
	
	In this paper we study the exact reconstruction of anisotropic inclusions in any spatial dimension $d\geq 2$. We reconstruct the outer shape $D^\bullet$ of
	\begin{equation*}
		D = \suppm(A_D - A_0),
	\end{equation*}
	from knowledge of $A_0$ and the local ND map $\Lambda(A_D)$ associated with the unknown conductivity~$A_D$. Here $D^\bullet$ is the smallest closed set with connected complement such that~$D\subseteq D^\bullet$. We give direct reconstruction methods (Theorems~\ref{thm:main}, \ref{thm:mainlinear}, and~\ref{thm:mainextreme}) based on the monotonicity of the forward map $A\mapsto\Lambda(A)$. The main assumptions are self-adjointness of the coefficients, that $A_0$ satisfies a unique continuation principle, and that $A_D-A_0$ satisfies a local uniform definiteness condition inside $D$ near $\partial D^\bullet$; see Assumption~\ref{assump:recon} for the precise statement. In particular, inside $D$ we do not assume more regularity than $L^\infty$ for $A_D$, and $A_D-A_0$ can even be indefinite at a positive distance away from~$\partial D^\bullet$.
	
	As a consequence, for a given self-adjoint ``background conductivity''~$A_0$ satisfying the unique continuation principle, and any self-adjoint conductivities $A_1$ and $A_2$ (that equal $A_0$ near $\partial\Omega$), we have that
	\begin{equation*}
		\Lambda(A_1) = \Lambda(A_2) \quad\text{\emph{and}}\quad \parbox{9.5em}{\emph{definiteness condition}\\ \emph{holds for $A_1$ and $A_2$}} \quad \Rightarrow \quad \suppm(A_1-A_0)^\bullet = \suppm(A_2-A_0)^\bullet.
	\end{equation*}  
	Hence, the obstructions from the anisotropic Calder\'on problem do not affect the inclusion detection problem under the definiteness condition. Practical reconstruction of inclusions is therefore sensible in a very high generality compared to reconstructing an anisotropic conductivity. 
	
	To provide some further intuition, consider a self-adjoint conductivity $A_D$, satisfying the definiteness condition, and with $\widetilde{A}_D$ defined from $A_D$ via \eqref{eq:pushforward} for a transformation $\Phi$. Let
	\begin{equation*}
		\widetilde{D} = \Phi(D)
	\end{equation*} 
	be the transformed inclusion under $\Phi$. If $\widetilde{D}^\bullet \neq D^\bullet$ then $\Phi$ will fail to be the identity in an open set intersecting $\partial D^\bullet$ (the part where it is deformed by $\Phi$), which also changes the background conductivity in $\widetilde{A}_D$ in that part of the domain via \eqref{eq:pushforward}. We conclude for
	\begin{equation*}
		\widehat{D} = \suppm(\widetilde{A}_D-A_0),
	\end{equation*} 
	if $\widehat{D}^\bullet \neq D^\bullet$, the definiteness condition cannot be satisfied inside $\widehat{D}$ near $\partial\widehat{D}^\bullet$. In particular, it may fail in part of the domain where the background conductivity in $\widetilde{A}_D$ is modified. 
	
	See also \cite{Crista2017} for a stability result in a particular geometric setting. Moreover for Helmholtz scattering, there are uniqueness results for inclusion detection with some anisotropic coefficients~\cite{Hahner2000,Cakoni2022}.
	
	\subsection{Outlining the main results}
	
	The reconstruction methods we present are based on monotonicity of the forward map via the following operator inequalities (Proposition~\ref{prop:mono}):
	\begin{equation*}
		\int_\Omega (A_2-A_1)\nabla u_2\cdot \overline{\nabla u_2}\,\di x \leq \inner{f,(\Lambda_1-\Lambda_2)f} \leq \int_\Omega A_2 A_1^{-1}(A_2-A_1)\nabla u_2\cdot\overline{\nabla u_2}\,\di x,
	\end{equation*}
	where $\Lambda_j = \Lambda(A_j)$ are the local ND maps, and $u_2$ is the electric potential for conductivity $A_2$ and Neumann condition~$f$. The definiteness properties of $A_2 A_1^{-1}(A_2-A_1)$ turn out to be the same as for $A_2-A_1$ (Proposition~\ref{prop:bnd}). 
	
	The \emph{outer approach} of the monotonicity method, given in our main results (Theorems~\ref{thm:main},  \ref{thm:mainlinear}, and~\ref{thm:mainextreme}), constructs test-operators that use the monotonicity to check if a given admissible test-set $C$ fully contains $D^\bullet$ or not. Reconstruction is of the form (from Theorem~\ref{thm:main}):
	\begin{equation*}
		D \subseteq C \quad \text{\emph{if and only if}} \quad \Lambda_C^{-} \geq \Lambda(A_D) \geq \Lambda_C^{+},
	\end{equation*}
	for test-operators $\Lambda_C^{-}$ and $\Lambda_C^{+}$, and the local ND map $\Lambda(A_D)$ for the unknown conductivity. In Theorem~\ref{thm:mainlinear} the test-operators are linearized, which is beneficial for numerical implementation. 
	
	For Theorem~\ref{thm:mainextreme} we also allow perfectly insulating and perfectly conducting parts in the unknown conductivity, and this is also used for the test-operators, which means that no conductivity bounds are needed for this approach. Theorem~\ref{thm:mainextreme} is arguably the overall main result of the paper, although with the downside of requiring Lipschitz regularity of $\partial D^\bullet$ (which is unlike in Theorems~\ref{thm:main} and~\ref{thm:mainlinear}, where no explicit boundary regularity is needed for the inclusions). 
	
	The paper is written in two parts, with Sections~\ref{sec:forward}--\ref{sec:mainnegproof} dedicated to the finite perturbation setting, while Sections~\ref{sec:forwardextreme}--\ref{sec:mainproofextreme} are dedicated to the extreme perturbation setting, following the approach for the isotropic case in~\cite{Garde2020,Garde2022b}. It should be noted that Theorem~\ref{thm:mainextreme} \emph{cannot} simply be obtained via a limit from Theorem~\ref{thm:main}. There is no short or straightforward proof for the difficult direction in the \emph{if and only if} result, and it requires the introduction of several tools specific for extreme anisotropic coefficients; see~\cite[remark 5.4]{Garde2020} for additional details in the isotropic case.
	
	The proofs rely on the theory of localized potentials~\cite{Gebauer2008b}, for the existence of certain localizing solutions to the conductivity equation, based on a unique continuation principle. We generalize the result on localized potentials to anisotropic (even non-self-adjoint) coefficients in Theorem~\ref{thm:locpot}.
	
	Under much stronger assumptions, we also give an \emph{inner approach} of the monotonicity method in Theorems~\ref{thm:mainlinearinnerpos}, \ref{thm:mainlinearinnerneg}, \ref{thm:extremelinear}, and~\ref{thm:extreme2}, which checks if a given open set $B$ is contained in $D^\bullet$.
	
	\subsection{Previous results for the monotonicity method}
	
	The monotonicity method is well-established for \emph{isotropic} conductivities, and for other similar inverse coefficient problems. We mention some of the results here, with a focus on the Calder\'on problem. In~\cite{Tamburrino2002} they use the monotonicity to give bounds on inclusions, and later~\cite{Harrach10,Harrach13} prove that it is an exact reconstruction method for finite perturbations to the background conductivity, using the theory of localized potentials. It also includes the case with both positive and negative perturbations, if lower and upper bounds are known for the perturbed conductivity. This is expanded upon in~\cite{Garde2020}, and now the perturbed conductivity can simultaneously have parts with finite positive and negative perturbations, as well as extreme parts that are perfectly insulating and perfectly conducting. The need for lower and upper bounds is also removed. Finally, in~\cite{GardeVogelius2024} the method is generalized to reconstruct collections of Lipschitz cracks.
	
	The method has also been used in other specialized directions: degenerate and singular perturbations based on $A_2$-Muckenhoupt weights~\cite{Garde2022b}, inclusions in a fractional Calder\'on problem~\cite{Lin2019,Lin2020}, inclusions in a $p$-Laplace equation~\cite{Brander2018}, for nonlinear materials~\cite{Tamburrino2025}, and for local uniqueness in an inverse coefficient problem for a non-resonant Schr\"odinger equation~\cite{Pohjola2019}. There are also rigorous connections to practically relevant electrode models~\cite{GardeStaboulis2016,GardeStaboulis2019,Harrach19,Harrach15}. 
	
	Certain reconstruction methods for the partial data isotropic Calder\'on problem also rely on the same type of methodology. For piecewise constant layered conductivities, there is an exact reconstruction method for both conductivity coefficient and the piecewise constant partition~\cite{Garde2020b,Garde2022}. In a finite-dimensional setting (where the partition is known), a general piecewise constant conductivity is reconstructed in \cite{Garde2025}, and likewise via a reformulation to a convex semidefinite optimization problem in~\cite{Harrach23}. 
	
	\subsection{Remarks on notation} \label{sec:notation}
	
	All considered vector spaces are complex; this includes the classes of conductivity coefficients. We use the convention that inner products on complex Hilbert spaces are linear in the first entry and anti-linear in the second. We denote the Euclidean inner product as $z_1\cdot\overline{z_2}$ for $z_1,z_2\in \C^d$, in particular the ``dot'' is bilinear. The Euclidean norm is denoted~$\abs{\,\cdot\,}$.
	
	For self-adjoint operators $A,B\in\mathscr{L}(H)$ for a Hilbert space $H$, we write $A\geq B$ when $A-B$ is positive semidefinite; this is the Loewner ordering of such operators. Moreover, for $A$ being a positive definite operator, $A^{1/2}$ denotes its unique positive definite square root.
	
	We consistently make use of the \emph{essential} support/supremum/infimum. We denote the interior of a set $C\subseteq \R^d$ as $C^\circ$. At certain occurrences (mainly from Section~\ref{sec:forwardextreme} and onwards) we denote by $K>0$ a generic constant that may change from line to line in the mathematical contents.
	
	\section{The forward problem} \label{sec:forward}
	
	Let $\Omega$ be a bounded Lipschitz domain in $\R^d$, for $d\in\N\setminus\{1\}$, with connected complement. For a bounded matrix-valued function $A\in L^\infty(\Omega)^{d\times d}$, we use the norm
	\begin{equation*}
		\norm{A}_{*} = \sup_{x\in\Omega}\norm{A(x)}_2,
	\end{equation*}
	where $\norm{\,\cdot\,}_2$ is the usual spectral norm for matrices, i.e.\ the operator norm on Euclidean space. Hence, we have
	\begin{equation*}
		\abs{A(x)\xi} \leq \norm{A}_*\abs{\xi}
	\end{equation*}
	for all $\xi\in\C^d$ and a.e.~$x\in\Omega$. 
	
	We use the notation 
	\begin{equation*}
		\AR = \frac{1}{2}\bigl(A + A^*\bigr) \quad \text{and} \quad \AI = \frac{1}{2\I}\bigl(A - A^*\bigr),
	\end{equation*}
	such that 
	\begin{equation*}
		A = \AR + \I\AI,
	\end{equation*}
	where we note that $\AR$ and $\AI$ are self-adjoint as matrices, i.e.\ Hermitian at every point in $\Omega$. This gives rise to the following general class of anisotropic coefficients:
	\begin{equation*}
		\mathcal{H}(\Omega) = \{\, A\in L^\infty(\Omega)^{d\times d} \mid \exists c>0 \colon \AR \geq c I \text{ in the Loewner order in $\Omega$} \,\}.
	\end{equation*}
	The condition on $A$, the uniform positive definiteness of $\AR$, can also explicitly be written:
	\begin{equation*}
			\inf_{x\in\Omega} \redel\bigl(A(x)\xi\cdot\overline{\xi}\bigr) = \inf_{x\in\Omega} \AR(x)\xi\cdot\overline{\xi} \geq c\abs{\xi}^2
	\end{equation*}
	for all $\xi\in\C^d$; there are other classes of strongly elliptic coefficients (when restricting to $\xi\in\R^d$) that we do not consider here, see e.g.~\cite[chapter 4]{McLean}. We also have the subset where $A$ is self-adjoint, i.e.\ when $A=\AR$:
	\begin{equation*}
		\Hsa(\Omega) = \{\, A\in\mathcal{H}(\Omega) \mid \AI = 0 \,\}.
	\end{equation*}
	
	Consider the partial data anisotropic conductivity problem for $A\in\mathcal{H}(\Omega)$: 
	\begin{equation} \label{eq:condeq}
		-\nabla\cdot(A\nabla u) = 0 \text{ in } \Omega, \qquad 	\nu\cdot(A\nabla u) = \begin{dcases}
			f & \text{on } \Gamma, \\
			0 & \text{on } \partial\Omega\setminus\Gamma.
		\end{dcases}
	\end{equation}
	Here $\nu$ is the outer unit normal to $\Omega$, $\Gamma\subseteq \partial \Omega$ is a non-empty open boundary piece, and the mean-free current density $f$ belongs to
	\begin{equation*}
		L^2_\diamond(\Gamma) = \{\, f\in L^2(\Gamma) \mid \inner{f,1} = 0 \,\}.
	\end{equation*}
	We denote by $\inner{\,\cdot\,,\,\cdot\,}$ the standard inner product on $L^2(\Gamma)$, and $\norm{\,\cdot\,}$ is the associated norm. We likewise denote
	\begin{equation*}
		H^1_\diamond(\Omega) = \{\, u\in H^1(\Omega) \mid \inner{u|_\Gamma,1} = 0 \,\}.
	\end{equation*}
	
	A Poincar\'e inequality, related to the mean-free condition in $H^1_\diamond(\Omega)$, enables that 
	\begin{equation*}
		\norm{v}_{H^1_\diamond(\Omega)} = \Bigl(\int_\Omega \abs{\nabla v}^2\,\di x \Bigr)^{1/2}
	\end{equation*}
	becomes a norm on $H^1_\diamond(\Omega)$, equivalent with the usual $H^1$-norm. Moreover, the following continuity and coercivity estimates hold for all $w,v\in H^1_\diamond(\Omega)$:
	\begin{align*}
		\bigl\lvert\int_\Omega A\nabla w \cdot\overline{\nabla v}\,\di x\bigr\rvert &\leq \norm{A}_*\norm{w}_{H^1_\diamond(\Omega)}\norm{v}_{H^1_\diamond(\Omega)}, \\[1mm]
		\bigl\lvert\int_\Omega A\nabla v \cdot\overline{\nabla v}\,\di x\bigr\rvert &\geq \redel\int_\Omega A\nabla v \cdot\overline{\nabla v}\,\di x \geq c\norm{v}_{H^1_\diamond(\Omega)}^2. 
	\end{align*}
	By the Lax--Milgram lemma, there exists a unique solution $u = u_f^A \in H^1_\diamond(\Omega)$ to the corresponding weak problem associated to \eqref{eq:condeq}:
	\begin{equation} \label{eq:weak}
		\int_\Omega A\nabla u \cdot\overline{\nabla v}\,\di x = \inner{f,v}
	\end{equation}
	for all $v\in H^1_\diamond(\Omega)$. We let $\Lambda(A) \in \mathscr{L}(L^2_\diamond(\Gamma))$ denote the associated local ND map,
	\begin{equation*}
		\Lambda(A)f = u_f^A|_{\Gamma},
	\end{equation*}
	and the forward problem is the nonlinear mapping $\Lambda \colon \mathcal{H}(\Omega) \to \mathscr{L}(L^2_\diamond(\Gamma))$. 
	
	From \eqref{eq:weak} it holds that for any $A_1,A_2\in \mathcal{H}(\Omega)$ and $f,g\in L^2_\diamond(\Gamma)$, we have
	\begin{equation} \label{eq:lambdaweak}
		\inner{f,\Lambda(A_1)g} = \int_\Omega A_2\nabla u_f^{A_2}\cdot \overline{\nabla u_g^{A_1}}\,\di x.
	\end{equation}
	Thus the adjoint operator satisfies
	\begin{equation*}
		\inner{f,\Lambda(A_1)^*g} = \overline{\inner{g,\Lambda(A_1)f}} = \int_\Omega \overline{A_2 \nabla u_g^{A_2}}\cdot\nabla u_f^{A_1}\,\di x = \int_\Omega A_2^*\nabla u_f^{A_1}\cdot\overline{\nabla u_g^{A_2}} \,\di x,
	\end{equation*}
	and therefore
	\begin{equation*}
		\inner{f,\Lambda(A^*)g} = \int_\Omega A\nabla u_f^A \cdot\overline{\nabla u_g^{A^*}}\,\di x = \inner{f,\Lambda(A)^*g}.
	\end{equation*}
	Specifically $\Lambda(A^*) = \Lambda(A)^*$, meaning that $\Lambda(A)$ is self-adjoint if $A\in\Hsa(\Omega)$. 
	
	Of particular interest are the quadratic forms
	\begin{equation*}
		\inner{f,\Lambda(A)f} = \int_\Omega A\nabla u_f^A\cdot\overline{\nabla u_f^A} \,\di x \quad \text{and} \quad 
		\inner{f,\Lambda(A)^*f} = \int_\Omega A^*\nabla u_f^A\cdot\overline{\nabla u_f^A} \,\di x.
	\end{equation*}
	We also write
	\begin{equation*}
		\LambdaR(A) = \frac{1}{2}\bigl(\Lambda(A) + \Lambda(A)^*\bigr) \quad \text{and} \quad \LambdaI(A) = \frac{1}{2\I}\bigl(\Lambda(A) - \Lambda(A)^*\bigr),
	\end{equation*}
	so for the forward problem we have
	\begin{equation*}
		\Lambda = \LambdaR + \I\LambdaI,
	\end{equation*}
	where $\LambdaR(A)$ and $\LambdaI(A)$ are self-adjoint operators. In terms of the quadratic forms we thus have
	\begin{align}
		\inner{f,\LambdaR(A)f} &= \redel\inner{f,\Lambda(A)f} = \int_\Omega \AR\nabla u_f^A\cdot\overline{\nabla u_f^A} \,\di x = \int_\Omega \abs{(\AR)^{1/2}\nabla u_f^A}^2\,\di x, \label{eq:lambdaR}\\
		\inner{f,\LambdaI(A)f} &= -\imdel\inner{f,\Lambda(A)f} = -\int_\Omega \AI\nabla u_f^A\cdot\overline{\nabla u_f^A} \,\di x. \label{eq:lambdaI}
	\end{align}
	As one would expect, when multiplying with a complex scalar $\kappa\in\C$, we get
	\begin{equation*}
		(\kappa\Lambda)\upR = \kappa\upR\LambdaR - \kappa\upI\LambdaI \quad \text{and} \quad (\kappa\Lambda)\upI = \kappa\upI\LambdaR + \kappa\upR\LambdaI.
	\end{equation*}
	
	\section{Outer reconstruction of general anisotropic inclusions} \label{sec:main}
	
	In the following we will assume that a so-called ``background conductivity'' $A_0\in\Hsa(\Omega)$ is known. For an unknown $A_D\in \Hsa(\Omega)$, we call 
	\begin{equation*}
		D = \suppm(A_D - A_0)
	\end{equation*}
	the \emph{inclusions}, noting that $D$ may have several connected components. Our reconstruction methods will relate to the so-called outer shape of $D$.
	\begin{definition}
		The \emph{outer shape} $D^\bullet$ of $D$ is the smallest closed set with connected complement such that~$D\subseteq D^\bullet$. 
	\end{definition}
	One of the key assumptions we will need, required for the existence of certain localized solutions, is that the background conductivity $A_0$ satisfies a unique continuation principle.
	\begin{definition} \label{defi:ucp}
		Let $U\subseteq \overline{\Omega}$ be relatively open and connected. $A\in \mathcal{H}(\Omega)$ satisfies the weak \emph{unique continuation principle} (UCP) in $U$ for the conductivity equation if only the trivial solution of 
		\begin{equation*}
			-\nabla \cdot(A\nabla u) = 0 \text{ in } U^\circ
		\end{equation*}
		can be identically zero in a non-empty open subset of $U$, and only the trivial solution has vanishing Cauchy data, $u$ and $\nu\cdot(A\nabla u)$, on a non-empty open part of $U\cap \partial\Omega$. If $U = \overline{\Omega}$ we just say that $A$ satisfies the UCP.
	\end{definition}	
	\begin{remark}
		It should be noted that the ``UCP from Cauchy data'' follows from the ``UCP from open subsets'', so for simplicity we are grouping both concepts together in Definition~\ref{defi:ucp}.
		
		In spatial dimension $d=2$, the UCP holds for general real-valued $A\in\Hsa(\Omega)$ (and even in some non-self-adjoint cases), see~\cite{Alessandrini2012}. 
		
		In spatial dimension $d \geq 3$, the UCP holds for real-valued Lipschitz continuous coefficients $A\in\Hsa(\Omega)$, see~\cite{Garofalo1986}, and with counterexamples to the UCP for $\alpha$-H\"older continuous coefficients with $\alpha\in(0,1)$, see~\cite{Miller74,Mandache1998}. 
		
		We note that there also exist UCP results related to classes of complex coefficients, e.g.~with \emph{principally normal symbol} \cite[chapter~28]{HormanderIV}.
	\end{remark}
	We define the class of admissible test-inclusions, used for the reconstruction methods, as
	\begin{equation*}
		\mathcal{A} = \{\, C \Subset \Omega \mid C \text{ is the closure of an open set and has connected complement} \,\}.
	\end{equation*}
	\begin{assumption}  \label{assump:recon} \needspace{2\baselineskip}  {}\
		\begin{enumerate}[\rm(i)]
			\item Assume there are \emph{known} bounds on $A_D$, i.e.\ scalars $0 < \alpha \leq \beta$ such that
			\begin{equation*}
				\alpha I \leq A_D \leq \beta I
			\end{equation*}
			in the Loewner order in $\Omega$.
			\item Assume that $D\Subset \Omega$ and is the closure of an open set.
			\item For every $x\in\partial D^\bullet$ and every open neighborhood $W$ of $x$, assume there exists a relatively open connected set $V\subset D^\bullet\cap W$ that intersects $\partial D^\bullet$, satisfying either of two options:
			\begin{enumerate}[\rm(a)]
				\item $A_D - A_0$ is positive semidefinite in $V$, and there exists an open ball $B\subset V$ on which $A_D - A_0$ is uniformly positive definite.
				\item $A_D - A_0$ is negative semidefinite in $V$, and there exists an open ball $B\subset V$ on which $A_D - A_0$ is uniformly negative definite.
			\end{enumerate}
			\item Assume that $A_0$ satisfies the UCP.
		\end{enumerate}
	\end{assumption}
	We give a few remarks on these assumptions.
	\begin{remark}[Remarks on Assumption~\ref{assump:recon}] {}\
		\begin{enumerate}[(a)]
			\item In case $\Gamma = \partial\Omega$, then $C\Subset\Omega$ in the definition of $\mathcal{A}$ can be replaced by $C\subseteq\overline{\Omega}$, and $D\Subset\Omega$ in Assumption~\ref{assump:recon}(ii) can be replaced by $D\subseteq\overline{\Omega}$. Also, here ``connected complement'' is in $\R^d$.
			\item Note that Assumption~\ref{assump:recon}(iii) only relates to what happens close to $\partial D^\bullet$, so further into $D$ we may have that $A_D-A_0$ is e.g.~indefinite. Moreover, the definiteness condition is not exactly at $\partial D^\bullet$ but just close to $\partial D^\bullet$. This allows for both a jump condition or a continuous deviation from $A_0$; in the isotropic case this mimics the condition of a strict local increase/decrease from the background conductivity when entering $D$ from the outside~\cite{Garde2022b}.
		\end{enumerate}
	\end{remark}
	For a measurable set $C$, we define the test-coefficients
	\begin{equation*}
		A^-_{C} = \begin{dcases}
			\alpha I & \text{in } C \\
			A_0 & \text{in } \Omega\setminus C
		\end{dcases}
		\quad 
		\text{and}
		\quad
		A^+_{C} = \begin{dcases}
			\beta I & \text{in } C \\
			A_0 & \text{in } \Omega\setminus C,
		\end{dcases}
	\end{equation*}
	and the corresponding test-operators
	\begin{equation*}
		\Lambda^-_C = \Lambda(A^-_C) \quad \text{and}\quad \Lambda^+_C = \Lambda(A^+_C).
	\end{equation*}
	\begin{theorem} \label{thm:main} \needspace{2\baselineskip} {}\
		Under Assumption~\ref{assump:recon}(i), for any measurable $C\subseteq \overline{\Omega}$ we have
		\begin{equation*}
			D \subseteq C \quad \text{implies} \quad \Lambda^-_C \geq \Lambda(A_D) \geq \Lambda^+_C.
		\end{equation*}
		Under all of Assumption~\ref{assump:recon}, for any $C\in\mathcal{A}$ we have
		\begin{equation*}
			\Lambda^-_C \geq \Lambda(A_D) \geq \Lambda^+_C \quad \text{implies} \quad D \subseteq C.
		\end{equation*}
	\end{theorem}
	\begin{proof}
		The proof is given in Section~\ref{sec:mainproof}.
	\end{proof}
	
	We also introduce linearized versions of the test-operators, which are more suitable for numerical computations. Note that for $A\in\Hsa(\Omega)$ and $B\in L^\infty(\Omega)^{d\times d}$, the Fr\'echet derivative of $\Lambda$ evaluated at $A$ and in direction $B$ is given by (c.f.~\cite{Garde2022c}):
	\begin{equation*}
		\inner{f,\DLambda(A;B)f} = -\int_\Omega B^*\nabla u_f^A\cdot\overline{\nabla u_f^A}\,\di x.
	\end{equation*}
	To shorten the notation in the following result, we introduce the test-operators
	\begin{align*}
		\DLambda^+_{C} &= \DLambda\bigl(A_0;(\beta I - A_0)\chi_C\bigr), \\
		\DLambda^-_{C} &= \DLambda\bigl(A_0;(A_0 - \tfrac{\beta^2}{\alpha}I)\chi_C\bigr), 		
	\end{align*}
	where $\chi_C$ is a characteristic function on the set $C$.
	\begin{theorem} \label{thm:mainlinear}
		Pick $0<\alpha\leq\beta$ such that Assumption~\ref{assump:recon}(i) holds for both $A_D$ and $A_0$. For any measurable $C\subseteq \overline{\Omega}$ we have
		\begin{equation*}
			D \subseteq C \quad \text{implies} \quad\DLambda^-_{C} \geq \Lambda(A_D) - \Lambda(A_0) \geq \DLambda^+_{C}.
		\end{equation*}
		Under all of Assumption~\ref{assump:recon}, for any $C\in\mathcal{A}$ we have
		\begin{equation*}
			\DLambda^-_{C} \geq \Lambda(A_D) - \Lambda(A_0) \geq \DLambda^+_{C} \quad \text{implies} \quad D \subseteq C.
		\end{equation*}
	\end{theorem}
	\begin{proof}
		The proof is given in Section~\ref{sec:mainlinearproof}.
	\end{proof}
	
	\begin{remark}
		Under the given assumptions, Theorem~\ref{thm:main} gives
		\begin{equation*}
			D^\bullet = \cap\, \{\, C\in\mathcal{A} \mid \Lambda^-_C \geq \Lambda(A_D) \geq \Lambda^+_C \,\},
		\end{equation*}
		while Theorem~\ref{thm:mainlinear} gives
		\begin{equation*}
			D^\bullet = \cap\, \{\, C\in\mathcal{A} \mid \DLambda^-_{C} \geq \Lambda(A_D)-\Lambda(A_0) \geq \DLambda^+_{C} \,\}.
		\end{equation*}
		Moreover, from the proofs we also conclude:
		\begin{itemize}
			\item If in Assumption~\ref{assump:recon}(iii) there is only positive definiteness near $\partial D^\bullet$, we only need to check the inequalities $\Lambda(A_D)\geq \Lambda^+_C$ or $\Lambda(A_D)-\Lambda(A_0) \geq \DLambda^+_{C}$.
			\item If in Assumption~\ref{assump:recon}(iii) there is only negative definiteness near $\partial D^\bullet$, we only need to check the inequalities $\Lambda^-_C \geq \Lambda(A_D)$ or $\DLambda^-_{C} \geq \Lambda(A_D)-\Lambda(A_0)$.
		\end{itemize} 
	\end{remark}
	
	\section{Inner reconstruction of uniformly definite anisotropic inclusions}
	
	Under much stronger definiteness assumptions (not just near $\partial D^\bullet$ as before), we can also use the inner approach of the monotonicity method. We focus on the linearized version, so that the UCP requirement is only related to the background conductivity $A_0$.
	\begin{assumption}  \label{assump:recon2} \needspace{2\baselineskip}  {}\
		\begin{enumerate}[\rm(i)]
			\item Assume that we have either:
			\begin{enumerate}[\rm(a)]
				\item \emph{Positive definite inclusions}, $A_D - A_0 \geq cI$ in the Loewner order in $D$, for known $c>0$.
				\item \emph{Negative definite inclusions}, $A_D - A_0 \leq -cI$ in the Loewner order in $D$, for known $c>0$.
			\end{enumerate}
			\item Assume that $D\Subset \Omega$ and is the closure of an open set.
			\item Assume that $A_0$ satisfies the UCP.
		\end{enumerate}
	\end{assumption}
	
	Again, to shorten the notation in the following results, we introduce the test-operators
	\begin{align*}
		\DLambdah^+_{B} &= \DLambda\bigl(A_0;c(\tfrac{\alpha}{\beta})^2\chi_B I\bigr), \\
		\DLambdah^-_{B} &= \DLambda\bigl(A_0;-c\chi_B I\bigr), 		
	\end{align*}
	where $\chi_B$ is a characteristic function on the set $B$; note the differences compared to the previous section.
	
	\begin{theorem} \label{thm:mainlinearinnerpos}
		Assume that $A_D$ has \emph{positive definite inclusions} in $D$, corresponding to case~(a) of Assumption~\ref{assump:recon2}(i). Let $A_0 \geq \alpha I$ and assume $A_D \leq \beta I$ in the Loewner order in $\Omega$, for $\alpha,\beta>0$. For any measurable $B\subseteq \Omega$ we have
		\begin{equation*}
			B \subseteq D \quad \text{implies} \quad \Lambda(A_D) - \Lambda(A_0) \leq \DLambdah^+_{B}.
		\end{equation*}
		Under all of Assumption~\ref{assump:recon2}, for any open set $B\subseteq \Omega$ we have
		\begin{equation*}
			\Lambda(A_D) - \Lambda(A_0) \leq \DLambdah^+_{B} \quad \text{implies} \quad B \subset D^\bullet.
		\end{equation*}
	\end{theorem}
	\begin{proof}
		The proof is given in Section~\ref{sec:mainposproof}.
	\end{proof}
	
	\begin{theorem} \label{thm:mainlinearinnerneg}
		Assume that $A_D$ has \emph{negative definite inclusions} in $D$, corresponding to case~(b) of Assumption~\ref{assump:recon2}(i). For any measurable $B\subseteq \Omega$ we have
		\begin{equation*}
			B \subseteq D \quad \text{implies} \quad \DLambdah^-_{B} \leq \Lambda(A_D) - \Lambda(A_0).
		\end{equation*}
		Under all of Assumption~\ref{assump:recon2}, for any open set $B\subseteq \Omega$ we have
		\begin{equation*}
			\DLambdah^-_{B} \leq \Lambda(A_D) - \Lambda(A_0) \quad \text{implies} \quad B \subset D^\bullet.
		\end{equation*}
	\end{theorem}
	\begin{proof}
		The proof is given in Section~\ref{sec:mainnegproof}.
	\end{proof}
	
	\begin{remark}
		Under the given assumptions, for positive definite inclusions, Theorem~\ref{thm:mainlinearinnerpos} gives
		\begin{equation*}
			D^\circ \subseteq \cup\, \{\, B\subseteq\Omega \text{ open ball} \mid \Lambda(A_D)-\Lambda(A_0) \leq \DLambdah^+_{B} \,\} \subset D^\bullet,
		\end{equation*}
		while for negative definite inclusions, Theorem~\ref{thm:mainlinearinnerneg} gives
		\begin{equation*}
			D^\circ \subseteq \cup\, \{\, B\subseteq\Omega \text{ open ball} \mid \DLambdah^-_{B} \leq \Lambda(A_D)-\Lambda(A_0) \,\} \subset D^\bullet.
		\end{equation*}
		One can of course replace ``open ball'' by some other basis of open sets for the Euclidean topology.
	\end{remark}
	
	\begin{remark}
		In the isotropic case, one can replace the factor $(\frac{\alpha}{\beta})^2$ in $\DLambdah^+_B$ by the larger factor~$\frac{\alpha}{\beta}$. This is because the same replacement holds in Proposition~\ref{prop:bnd}(ii) for isotropic coefficients, see e.g.~\cite[example~4.4]{Harrach13}.
	\end{remark}	
	
	\section{Operator inequalities for Neumann-to-Dirichlet maps} \label{sec:monoineq}
	
	Although we will only consider coefficients from $\Hsa(\Omega)$ in the reconstruction methods, we still prove the following more general inequalities based on $\mathcal{H}(\Omega)$ for possible future reference. See also~\cite[Lemma~2.1]{Harrach2015} for the isotropic case with complex conductivity.
	
	\begin{theorem} \label{thm:generalmono}
		Let $A_1,A_2\in\mathcal{H}(\Omega)$ and $f\in L^2_\diamond(\Gamma)$. For $j\in\{1,2\}$ we denote $\Lambda_j = \Lambda(A_j)$ and $u_j = u_f^{A_j}$. Assume $\kappa\in\C$ is such that $\kappa A_1 \in \mathcal{H}(\Omega)$. We also define
		\begin{align*}
			\mathcal{B}_\kappa &= (\kappa A_1)\upI[(\kappa A_1)\upR]^{-1}(\kappa A_1)\upI, \\
			\mathcal{C}_\kappa &=  (\kappa A_2)\upI[(\kappa A_1)\upR]^{-1}(\kappa A_2)\upI + \I\bigl[ (\kappa A_2)\upR[(\kappa A_1)\upR]^{-1}(\kappa A_2)\upI - (\kappa A_2)\upI[(\kappa A_1)\upR]^{-1}(\kappa A_2)\upR \bigr].
		\end{align*}
		Then
		\begin{align}
			\inner{f,[\overline{\kappa}(\Lambda_1-\Lambda_2)]\upR f} &\geq \int_\Omega \bigl[[\kappa(A_2-A_1)]\upR - \mathcal{B}_\kappa \bigr]\nabla u_2\cdot \overline{\nabla u_2}\,\di x, \label{eq:mono1} \\
			\inner{f,[\overline{\kappa}(\Lambda_1-\Lambda_2)]\upR f} &\leq \int_\Omega \bigl[(\kappa A_2)\upR [(\kappa A_1)\upR]^{-1}[\kappa(A_2-A_1)]\upR + \mathcal{C}_\kappa\bigr]\nabla u_2\cdot\overline{\nabla u_2}\,\di x. \label{eq:mono2}
		\end{align}
	\end{theorem}
	\begin{proof}
		We start by proving \eqref{eq:mono1}. A computation reveals that
		\begin{equation*}
			(\kappa A_1)[(\kappa A_1)\upR]^{-1}(\kappa A_1)^* = (\kappa A_1)\upR + \mathcal{B}_\kappa,
		\end{equation*}
		which is used below, combined with \eqref{eq:lambdaweak} and \eqref{eq:lambdaR}:
		\begin{align*}
			0 &\leq \int_\Omega \bigl\lvert[(\kappa A_1)\upR]^{1/2}\nabla u_1 - [(\kappa A_1)\upR]^{-1/2}(\kappa A_1)^*\nabla u_2\bigr\rvert^2 \,\di x \\
			&= \redel\Bigl( \kappa\int_\Omega A_1\nabla u_1\cdot\overline{\nabla u_1}\,\di x \Bigr) -2\redel\Bigl( \kappa\int_\Omega A_1\nabla u_1\cdot\overline{\nabla u_2}\,\di x \Bigr) \\
			&\phantom{={}}  + \int_\Omega \bigl[ (\kappa A_1)\upR + \mathcal{B}_\kappa \bigr]\nabla u_2\cdot\overline{\nabla u_2}\,\di x \\
			&= \inner{f,[\overline{\kappa}(\Lambda_1-\Lambda_2)]\upR f} +\int_\Omega \bigl[ (\kappa A_1)\upR - (\kappa A_2)\upR + \mathcal{B}_\kappa \bigr]\nabla u_2\cdot\overline{\nabla u_2}\,\di x.
		\end{align*}
		
		Next we prove \eqref{eq:mono2}. A computation reveals that
		\begin{align*}
			(\kappa A_2)^*[(\kappa A_1)\upR]^{-1}(\kappa A_2) &= (\kappa A_2)\upR [(\kappa A_1)\upR]^{-1}(\kappa A_2)\upR + \mathcal{C}_\kappa \\
			&=   (\kappa A_2)\upR + (\kappa A_2)\upR[(\kappa A_1)\upR]^{-1}[\kappa(A_2-A_1)]\upR + \mathcal{C}_\kappa,
		\end{align*}
		which is used below, combined with \eqref{eq:lambdaweak} and \eqref{eq:lambdaR}:
		\begin{align*}
			0 &\leq \int_\Omega \bigl\lvert[(\kappa A_1)\upR]^{1/2}\nabla u_1 - [(\kappa A_1)\upR]^{-1/2}(\kappa A_2)\nabla u_2\bigr\rvert^2 \,\di x \\
			&= \redel\Bigl( \kappa\int_\Omega A_1\nabla u_1\cdot\overline{\nabla u_1}\,\di x \Bigr) -2\redel\Bigl( \kappa\int_\Omega A_2\nabla u_2\cdot\overline{\nabla u_1}\,\di x \Bigr) \\
			&\phantom{={}}  + \int_\Omega \bigl[ (\kappa A_2)\upR + (\kappa A_2)\upR[(\kappa A_1)\upR]^{-1}[\kappa(A_2-A_1)]\upR + \mathcal{C}_\kappa \bigr]\nabla u_2\cdot\overline{\nabla u_2}\,\di x \\
			&= \inner{f,[\overline{\kappa}(\Lambda_2-\Lambda_1)]\upR f} +\int_\Omega \bigl[ (\kappa A_2)\upR[(\kappa A_1)\upR]^{-1}[\kappa(A_2-A_1)]\upR + \mathcal{C}_\kappa \bigr]\nabla u_2\cdot\overline{\nabla u_2}\,\di x. \qedhere
		\end{align*}
	\end{proof}
	
	\begin{remark}
		Note that in the isotropic case, the large square bracket in $\mathcal{C}_\kappa$ vanishes (after the imaginary unit). Also, if $\kappa A_1\in\Hsa(\Omega)$ then $\mathcal{B}_\kappa = 0$ and if $\kappa A_2\in\Hsa(\Omega)$ then $\mathcal{C}_\kappa = 0$.
	\end{remark}
	
	We get some corollaries of Theorem~\ref{thm:generalmono}, based on the choice of $\kappa$.
	
	\begin{corollary}
		In the setting of Theorem~\ref{thm:generalmono} with $\kappa = 1$, we have
		\begin{align*}
			\inner{f,(\LambdaR_1-\LambdaR_2) f} &\geq \int_\Omega \bigl[\AR_2-\AR_1 - \mathcal{B}_1 \bigr]\nabla u_2\cdot \overline{\nabla u_2}\,\di x,  \\
			\inner{f,(\LambdaR_1-\LambdaR_2) f} &\leq \int_\Omega \bigl[\AR_2 (\AR_1)^{-1}(\AR_2-\AR_1) + \mathcal{C}_1\bigr]\nabla u_2\cdot\overline{\nabla u_2}\,\di x. 
		\end{align*}
	\end{corollary}
	
	\begin{corollary}
		In the setting of Theorem~\ref{thm:generalmono} with $\kappa = -\I$, provided that $\AI_1 \in\Hsa(\Omega)$, we have
		\begin{align*}
			\inner{f,(\LambdaI_2-\LambdaI_1) f} &\geq \int_\Omega \bigl[\AI_2-\AI_1 - \mathcal{B}_{-\I} \bigr]\nabla u_2\cdot \overline{\nabla u_2}\,\di x,  \\
			\inner{f,(\LambdaI_2-\LambdaI_1) f} &\leq \int_\Omega \bigl[\AI_2 (\AI_1)^{-1}(\AI_2-\AI_1) + \mathcal{C}_{-\I}\bigr]\nabla u_2\cdot\overline{\nabla u_2}\,\di x. 
		\end{align*}
	\end{corollary}

	What we will need for the monotonicity method is the following, corresponding to the special case of Theorem~\ref{thm:generalmono} where $A_1,A_2\in\Hsa(\Omega)$ and $\kappa = 1$. We also provide a slightly different proof that is arguably easier to read. 

	\begin{proposition} \label{prop:mono}
		Let $A_1,A_2\in\Hsa(\Omega)$ and $f\in L^2_\diamond(\Gamma)$. For $j\in\{1,2\}$ we denote $\Lambda_j = \Lambda(A_j)$ and $u_j = u_f^{A_j}$. Then
		\begin{align}
			\inner{f,(\Lambda_1-\Lambda_2)f} &\geq \int_\Omega (A_2-A_1)\nabla u_2\cdot \overline{\nabla u_2}\,\di x, \label{eq:monosimple1} \\
			\inner{f,(\Lambda_1-\Lambda_2)f} &\leq \int_\Omega A_2 A_1^{-1}(A_2-A_1)\nabla u_2\cdot\overline{\nabla u_2}\,\di x. \label{eq:monosimple2}
		\end{align}
	\end{proposition}
	\begin{proof}
		We start by proving \eqref{eq:monosimple1}, which is by the computation below combined with \eqref{eq:lambdaweak},
		\begin{align*}
			&\int_\Omega (A_2 - A_1)\nabla u_2\cdot\overline{\nabla u_2}\,\di x \\
			&\hspace{1.5cm}\leq \int_\Omega (A_2 - A_1)\nabla u_2\cdot\overline{\nabla u_2}\,\di x + \int_\Omega \abs{A_1^{1/2}\nabla(u_1-u_2)}^2\,\di x \\
			&\hspace{1.5cm}= \int_\Omega A_2\nabla u_2\cdot\overline{\nabla u_2}\,\di x + \int_\Omega A_1\nabla u_1\cdot\overline{\nabla u_1}\,\di x - 2\redel\Bigl( \int_\Omega A_1\nabla u_1\cdot\overline{\nabla u_2}\,\di x \Bigr) \\
			&\hspace{1.5cm}=\inner{f,(\Lambda_1-\Lambda_2)f}.
		\end{align*}
		
		Next we prove \eqref{eq:monosimple2}. Using that 
		\begin{equation*}
			A_2A_1^{-1}(A_2-A_1) = A_2A_1^{-1}A_2 - A_2
		\end{equation*}
		is self-adjoint, we get
		\begin{align*}
			&\int_\Omega A_2 A_1^{-1}(A_2-A_1)\nabla u_2\cdot\overline{\nabla u_2}\,\di x \\
			&\hspace{1.5cm}\geq \int_\Omega A_2 A_1^{-1}(A_2-A_1)\nabla u_2\cdot\overline{\nabla u_2}\,\di x - \int_\Omega \abs{A_1^{1/2}\nabla u_1 - A_1^{-1/2}A_2\nabla u_2}^2\,\di x \\
			&\hspace{1.5cm}= - \int_\Omega A_2\nabla u_2\cdot\overline{\nabla u_2}\,\di x -\int_\Omega A_1\nabla u_1\cdot\overline{\nabla u_1}\,\di x + 2\redel\Bigl(\int_\Omega A_2\nabla u_2\cdot\overline{\nabla u_1}\,\di x\Bigr)  \\
			&\hspace{1.5cm}= \inner{f,(\Lambda_1 - \Lambda_2)f}. \qedhere
		\end{align*}
	\end{proof}
	
	Note that the first inequality in Proposition~\ref{prop:mono} implies monotonicity of the ND maps, since
	\begin{equation} \label{eq:simplemono}
		\int_\Omega (A_2-A_1)\nabla u_2\cdot \overline{\nabla u_2}\,\di x \leq \inner{f,(\Lambda_1-\Lambda_2)f} \leq \int_\Omega (A_2-A_1)\nabla u_1\cdot \overline{\nabla u_1}\,\di x.
	\end{equation}
	It is not immediately clear how one would make use of the second inequality in Proposition~\ref{prop:mono}, due to the matrix product in the front. What we need to consider is when $A_2-A_1$ is positive/negative (semi)definite, if we can also infer that $A_2A_1^{-1}(A_2-A_1)$ has the same properties. This is the result of the following proposition.
	\begin{proposition} \label{prop:bnd}
		Let $c\geq 0$ and $V \subseteq \Omega$. Let $A_1,A_2\in\Hsa(\Omega)$ with $A_1 \leq \beta I$ and $A_2 \geq \alpha I$  in $V$ for $\alpha,\beta>0$. 
		\begin{enumerate}[\rm(i)]
			\item If $A_2-A_1 \geq cI$ in $V$, then $A_2A_1^{-1}(A_2-A_1) \geq cI$ in $V$.
			\item If $A_2-A_1 \leq -cI$ in $V$, then $A_2A_1^{-1}(A_2-A_1) \leq -c(\tfrac{\alpha}{\beta})^2I$ in $V$.
		\end{enumerate}
	\end{proposition}
	\begin{proof}
		Let $A_2 - A_1 \geq cI$ in the Loewner order in $V$ for some $c\geq 0$. Let $\xi\in\C^d$ and note that $A_1^{-1}$ is positive definite. Now we put the pieces together:
		\begin{align*}
			A_2A_1^{-1}(A_2-A_1)\xi\cdot\overline{\xi} &= (A_2-A_1)A_1^{-1}(A_2-A_1)\xi\cdot\overline{\xi} + (A_2-A_1)\xi\cdot\overline{\xi} \\
			&= A_1^{-1}(A_2-A_1)\xi\cdot\overline{(A_2-A_1)\xi} + (A_2-A_1)\xi\cdot\overline{\xi} \\
			&\geq c\abs{\xi}^2,
		\end{align*}
		which concludes the proof of (i).
		
		For (ii), we first prove the assertion that if
		\begin{equation} \label{eq:negdefproof}
			A_1 - A_2 \geq cI
		\end{equation}
		for some $c\geq 0$, then
		\begin{equation} \label{eq:inverseposdef}
			A_2^{-1}-A_1^{-1} \geq c\beta^{-2}I.
		\end{equation}
		The semidefinite case ($c=0$) of \eqref{eq:inverseposdef} is a simple consequence of $t\mapsto -t^{-1}$ being an operator monotone function, c.f.~\cite[chapter~7]{Donoghue1974}. To prove \eqref{eq:inverseposdef}, first note that we have
		\begin{align*}
			A_2^{-1} - A_1^{-1} &= A_1^{-1}(A_1-A_2)A_2^{-1} \\
			&=A_1^{-1}(A_1-A_2)A_1^{-1} + A_1^{-1}(A_1-A_2)(A_2^{-1}-A_1^{-1}) \\
			&= A_1^{-1}\bigl[ A_1-A_2 + (A_1-A_2)A_2^{-1}(A_1-A_2) \bigr]A_1^{-1}.
		\end{align*}
		Note that $A_1^{-1} \geq \beta^{-1}I$ and likewise that $A_2^{-1}$ is positive definite. Hence, we have
		\begin{align*}
			(A_2^{-1} - A_1^{-1})\xi\cdot\overline{\xi} &= (A_1-A_2)A_1^{-1}\xi\cdot\overline{A_1^{-1}\xi} + A_2^{-1}(A_1-A_2)A_1^{-1}\xi\cdot\overline{(A_1-A_2)A_1^{-1}\xi} \\
			&\geq c\abs{A_1^{-1}\xi}^2 \geq c\beta^{-2}\abs{\xi}^2. 
		\end{align*}
			
		Now assume that in $V$ we have \eqref{eq:negdefproof} for some $c\geq 0$. Then \eqref{eq:inverseposdef} implies
		\begin{equation*}
			A_2A_1^{-1}(A_1-A_2)\xi\cdot\overline{\xi} = (A_2^{-1}-A_1^{-1})A_2\xi\cdot\overline{A_2\xi} \geq c\beta^{-2}\abs{A_2\xi}^2 \geq c(\tfrac{\alpha}{\beta})^2\abs{\xi}^2,
		\end{equation*}
		which concludes the proof of (ii).
	\end{proof}
	
	\section{Localized potentials for anisotropic coefficients} \label{sec:locpot}
	
	We extend localization results from \cite{Gebauer2008b} to anisotropic coefficients in $\mathcal{H}(\Omega)$. For a measurable set $V\subseteq \overline{\Omega}$ and $F\in L^2(V)^d$, we define $w = w_F^A$ as the unique solution in $H^1_\diamond(\Omega)$ to the variational problem
	\begin{equation} \label{eq:vmo_weak}
		\int_\Omega A^* \nabla w \cdot \overline{\nabla v}\,\di x = \int_V F\cdot\overline{\nabla v}\,\di x
	\end{equation}
	for all $v\in H^1_\diamond(\Omega).$ We define the \emph{virtual measurement operator} $L_V(A) \colon L^2(V)^d \to L^2_\diamond(\Gamma)$ as 
	\begin{equation*}
		L_V(A)F = w_F^A|_{\Gamma}.
	\end{equation*}
	In the following, $R$ is used to denote the range of an operator.
	\begin{proposition} \label{prop:vmo}
		For $A\in \mathcal{H}(\Omega)$ and measurable sets $V, V_1, V_2 \subseteq \overline{\Omega}$, there are the following properties of the virtual measurement operator:
		\begin{enumerate}[\rm(i)]
			\item $R(L_V(A))$ is independent of the function values $A|_V$,
			\item $V_1\subseteq V_2$ implies $R(L_{V_1}(A)) \subseteq R(L_{V_2}(A))$,
			\item The adjoint $L_V(A)^* \colon L^2_\diamond(\Gamma) \to L^2(V)^d$ is given by $L_V(A)^* f = \nabla u_f^A|_V$.
		\end{enumerate}
	\end{proposition}
	\begin{proof} \needspace{\lineskip} {}\		
		
		(i): Let $A_1,A_2\in \mathcal{H}(\Omega)$ such that $A_1 = A_2$ outside $V$. We will prove that $R(L_V(A_1)) = R(L_V(A_2))$. Let $\varphi\in R(L_V(A_1))$, i.e.\ $\varphi = w|_\Gamma$ with $w = w_{F_1}^{A_1}$ for some $F_1\in L^2(V)^d$. Now define $F_2 = (A_2-A_1)^*\nabla w|_V \in L^2(V)^d$. Then \eqref{eq:vmo_weak} gives
		\begin{align*}
			\int_\Omega A_2^*\nabla w\cdot\overline{\nabla v}\,\di x &= \int_\Omega (A_2-A_1)^*\nabla w\cdot\overline{\nabla v}\,\di x + \int_\Omega A_1^*\nabla w\cdot\overline{\nabla v}\,\di x \\
			&= \int_V (A_2-A_1)^*\nabla w\cdot\overline{\nabla v}\,\di x + \int_V F_1\cdot\overline{\nabla v}\,\di x \\
			&= \int_V (F_1+F_2)\cdot\overline{\nabla v}\, \di x
		\end{align*}
		for all $v\in H^1_\diamond(\Omega)$. Hence, we have that $\varphi = L_V(A_2)(F_1+F_2)$ so indeed $\varphi\in R(L_V(A_2))$. Repeating the above proof with $A_1$ and $A_2$ interchanged results in $R(L_V(A_1)) = R(L_V(A_2))$ as desired.
		
		(ii): Assume $V_1\subseteq V_2$, and let $\varphi = L_{V_1}(A)F = w|_\Gamma$ for some $F\in L^2(V_1)^d$. Let $\widetilde{F}$ denote the extension by zero of $F$ on the set $V_2$. Then
		\begin{equation*}
			\int_\Omega A^*\nabla w\cdot \overline{\nabla v}\,\di x = \int_{V_1} F\cdot\overline{\nabla v}\,\di x = \int_{V_2} \widetilde{F}\cdot\overline{\nabla v}\,\di x
		\end{equation*}
		for all $v\in H^1_\diamond(\Omega)$. Hence, we also have $\varphi = L_{V_2}(A)\widetilde{F}$, i.e.\ $R(L_{V_1}(A))\subseteq R(L_{V_2}(A))$.
		
		(iii): Let $F\in L^2(V)^d$ and $f\in L^2_\diamond(\Gamma)$, then from the variational forms \eqref{eq:vmo_weak} and \eqref{eq:weak}, for $w = w_F^A$ and $u = u_f^A$, we obtain
		\begin{align*}
			\inner{L_V(A)^* f, F}_{L^2(V)^d} &= \inner{f,L_V(A)F}_{L^2(\Gamma)} = \inner{f,w}_{L^2(\Gamma)} \\
			&= \int_\Omega A \nabla u \cdot\overline{\nabla w}\,\di x = \int_V \overline{F}\cdot \nabla u\,\di x = \inner{\nabla u|_V,F}_{L^2(V)^d}.
		\end{align*}
		Since this holds for all $F\in L^2(V)^d$, we conclude that $L_V(A)^* f = \nabla u|_V$.
	\end{proof}
	
	We note that range inclusions imply norm inequalities for the adjoint operators.
	\begin{lemma}[Lemma~2.5 in \cite{Gebauer2008b}] \label{lemma:rangenorm}
		Let $H$, $K_1$, and $K_2$ be Hilbert spaces and let $T_j\in\mathscr{L}(K_j,H)$ for $j \in \{1,2\}$. Then
		\begin{equation*}
			R(T_1) \subseteq R(T_2) \qquad \text{if and only if} \qquad \exists C>0,\, \forall x\in H\colon \norm{T_1^* x}_{K_1} \leq C \norm{T_2^* x}_{K_2}.
		\end{equation*}
	\end{lemma}
	
	Finally, we are able to prove the existence of localized potentials.
	
	\begin{theorem} \label{thm:locpot}
		Let $U\subset \overline{\Omega}$ be a relatively open connected set that intersects $\Gamma$ and let $B\subset U$ be a non-empty open set. Assume $A\in \mathcal{H}(\Omega)$ where both $A$ and $A^*$ satisfy the UCP in $U$. Then there are sequences $(f_j)$ in $L^2_\diamond(\Gamma)$ and $(u_j)$ in $H^1_\diamond(\Omega)$, with $u_j = u_{f_j}^A$, such that
		\begin{equation*}
			\lim_{j\to\infty} \int_B \abs{\nabla u_j}^2\,\di x = \infty \qquad \text{and} \qquad \lim_{j\to\infty}\int_{\Omega\setminus U} \abs{\nabla u_j}^2\,\di x = 0. 
		\end{equation*}
	\end{theorem}
	\begin{proof}
		We let $V_1 \Subset B \subset V \subseteq U$ where $V_1$ and $V$ are relatively open non-empty connected sets, and define $V_2 = \Omega\setminus\overline{V}$. These sets are chosen such that $V$ intersects $\Gamma$ and that $\partial (\Omega\setminus \overline{V_j})$ are Lipschitz continuous.
		
		We start by proving that $R(L_{V_1}(A))$ and $R(L_{V_2}(A))$ have a trivial intersection. Hence, let $\varphi \in R(L_{V_1}(A)) \cap R(L_{V_2}(A))$, i.e.\ $\varphi = w_j|_\Gamma$ with $w_j = w_{F_j}^A$ for some $F_j\in L^2(V_j)^d$ for $j\in\{1,2\}$. From \eqref{eq:vmo_weak} and the weak form of the conductivity problem \eqref{eq:weak}, with conductivity $A^*$, we thus have
		\begin{align*}
			-\nabla\cdot (A^*\nabla w_j) &= 0 \text{ in } \Omega\setminus \overline{V_j} \\
			 \nu\cdot(A^*\nabla w_j) &= 0 \text{ on } (\partial\Omega)\setminus \overline{V_j}.
		\end{align*}
		In particular $\nu\cdot(A^*\nabla w_j) \equiv 0$ on $\partial V \cap \Gamma$, i.e.\ the Cauchy data of $w_1$ and $w_2$ coincide on $\partial V\cap \Gamma$. Unique continuation entails that $w_1 = w_2$ in $V\setminus\overline{V_1}$. We may now ``glue'' $w_1$ and $w_2$ together, using that they agree in $V\setminus\overline{V_1}$:
		\begin{equation*}
			u = \begin{dcases}
				w_1 & \text{in } \Omega\setminus\overline{V_1}, \\
				w_2 & \text{in } V_1.
			\end{dcases}
		\end{equation*}
		Now we have $-\nabla\cdot(A^*\nabla u) = 0$ in $\Omega$ and $\nu\cdot (A^*\nabla u) = \nu\cdot(A^*\nabla w_1) \equiv 0$ on $\partial\Omega$. By uniqueness of such a solution in $H^1_\diamond(\Omega)$, it follows that $u\equiv 0$ in $\Omega$, so in particular $\varphi = w_1|_\Gamma = u|_\Gamma \equiv 0$, i.e.
		\begin{equation*}
			R(L_{V_1}(A)) \cap R(L_{V_2}(A)) = \{0\}.
		\end{equation*}
		
		Next we prove that $R(L_{V_1}(A))$ is non-trivial. Suppose that $0\equiv L_{V_1}(A)^* f = \nabla u_f^A|_{V_1}$, then $u_f^A$ is a constant function in $V_1$, and also in $V$ by unique continuation. In particular, $f = \nu\cdot (A\nabla u_f^A)|_\Gamma \equiv 0$ on $\partial V\cap \Gamma$. Thus, $L_{V_1}(A)^*$ is injective on
		\begin{equation*}
			W = \{\, f\in L^2_\diamond(\Gamma) \mid \suppm f \subseteq \partial V\cap\Gamma\,\}.
		\end{equation*}
		Since $\overline{R(L_{V_1}(A))} = N(L_{V_1}(A)^*)^\perp$ then $R(L_{V_1}(A))$ contains a dense subset of $W$ and must be non-trivial.
		
		In total we have that $R(L_{V_1}(A)) \not\subseteq R(L_{V_2}(A))$. From Lemma~\ref{lemma:rangenorm} this implies the existence of a sequence $(f_j)$ in $L^2_\diamond(\Gamma)$ such that $u_j = u_{f_j}^A$ satisfy (c.f.\ Proposition~\ref{prop:vmo}(iii)):
		\begin{equation*}
			\int_B \abs{\nabla u_j}^2\,\di x \geq \int_{V_1} \abs{\nabla u_j}^2\,\di x = \norm{L_{V_1}(A)^* f_j}_{L^2(V_1)^d}^2 \to \infty
		\end{equation*}
		and
		\begin{equation*}
			\int_{\Omega\setminus U} \abs{\nabla u_j}^2\,\di x \leq \int_{V_2} \abs{\nabla u_j}^2\,\di x = \norm{L_{V_2}(A)^* f_j}_{L^2(V_2)^d}^2 \to 0
		\end{equation*}
		for $j\to\infty$.
	\end{proof}
	
	The following result shows that the sequence $(f_j)$ of Neumann conditions (e.g.\ from Theorem~\ref{thm:locpot}) such that the corresponding power $\abs{\nabla u_j}^2$ is localized, can be chosen independently of the function values of $A$ on the set where the power tends to zero.
	
	\begin{proposition} \label{prop:simultloc}
		Let $B$ and $U$ be as in Theorem~\ref{thm:locpot}, and assume for $A_1\in \mathcal{H}(\Omega)$ and $(f_j)$ in $L^2_\diamond(\Gamma)$ that $u_j = u_{f_j}^{A_1}$ satisfy
		\begin{equation*}
			\lim_{j\to\infty} \int_B \abs{\nabla u_j}^2\,\di x = \infty \qquad \text{and} \qquad \lim_{j\to\infty}\int_{\Omega\setminus U} \abs{\nabla u_j}^2\,\di x = 0.
		\end{equation*}
		If $A_2\in \mathcal{H}(\Omega)$ with $\suppm(A_1-A_2)\subseteq \overline{\Omega}\setminus U$, then $\widehat{u}_j = u_{f_j}^{A_2}$ also satisfy
		\begin{equation*}
			\lim_{j\to\infty} \int_B \abs{\nabla \widehat{u}_j}^2\,\di x = \infty \qquad \text{and} \qquad \lim_{j\to\infty}\int_{\Omega\setminus U} \abs{\nabla \widehat{u}_j}^2\,\di x = 0.
		\end{equation*}
	\end{proposition}
	\begin{proof}
		Let $V_1 = B$ and $V_2 = \overline{\Omega}\setminus U$. From Proposition~\ref{prop:vmo}(iii) it amounts to proving
		\begin{equation*}
			\lim_{j\to\infty} \norm{L_{V_1}(A_2)^* f_j}_{L^2(V_1)^d} = \infty \quad\text{and}\quad \lim_{j\to\infty} \norm{L_{V_2}(A_2)^* f_j}_{L^2(V_2)^d} = 0.´
		\end{equation*}
		From Proposition~\ref{prop:vmo}(i) we have that $R(L_{V_2}(A_1)) = R(L_{V_2}(A_2))$. Hence a combination of Proposition~\ref{prop:vmo}(iii) and Lemma~\ref{lemma:rangenorm} gives
		\begin{equation*}
			\norm{L_{V_2}(A_2)^* f_j}_{L^2(V_2)^d} \leq C\norm{L_{V_2}(A_1)^* f_j}_{L^2(V_2)^d} \to 0 \text{ for } j\to\infty.
		\end{equation*}
		By the same argumentation, we also have $R(L_{V_1\cup V_2}(A_1)) = R(L_{V_1 \cup V_2}(A_2))$, which implies
		\begin{align*}
			\norm{L_{V_1 \cup V_2}(A_2)^* f_j}_{L^2(V_1\cup V_2)^d} &\geq c\norm{L_{V_1 \cup V_2}(A_1)^* f_j}_{L^2(V_1\cup V_2)^d} \\
			&\geq c\norm{L_{V_1}(A_1)^* f_j}_{L^2(V_1)^d} \to \infty \text{ for } j\to\infty.
		\end{align*}
		Finally we combine the two above limits, using Proposition~\ref{prop:vmo}(iii),
		\begin{equation*}
			\norm{L_{V_1}(A_2)^* f_j}_{L^2(V_1)^d}^2 = \norm{L_{V_1 \cup V_2}(A_2)^* f_j}_{L^2(V_1\cup V_2)^d}^2 - \norm{L_{V_2}(A_2)^* f_j}_{L^2(V_2)^d}^2 \to \infty \text{ for } j\to\infty,
		\end{equation*}
		thereby concluding the proof.
	\end{proof}
	
	\section{Proof of Theorem~\ref{thm:main}} \label{sec:mainproof}
	
	The proof that $D \subseteq C$ implies $\Lambda^-_C \geq \Lambda(A_D) \geq \Lambda^+_C$ follows directly from Assumption~\ref{assump:recon}(i) and~\eqref{eq:simplemono}, since we have
	\begin{equation*}
		A^-_{C} \leq A_D \leq A^+_{C}
	\end{equation*}
	in the Loewner order in $\Omega$.
	
	Now assume that $D \not\subseteq C$, i.e.~that $D^\bullet\not\subseteq C$. Since both sets are closures of opens sets and have connected complements, there is a relatively open set that connects $D^\bullet\setminus C$ with $\Gamma$ that we may use for localization. Specifically, there exist a relatively open connected set $U\subset \overline{\Omega}\setminus C$ that intersects $\Gamma$ and an open ball $B\subset U \cap D$. Moreover, Assumption~\ref{assump:recon}(iii) ensures that we can pick $B$ and $U$ such that we arrive at one of two cases:
	\begin{itemize}
		\item Case 1: $A_D - A_0$ is positive semidefinite in $U$ and uniformly positive definite in $B$.
		\item Case 2: $A_D - A_0$ is negative semidefinite in $U$ and uniformly negative definite in $B$.
	\end{itemize}
	In the first case we will prove that $\Lambda(A_D) \not\geq \Lambda^+_C$, and in the second case we will prove that $\Lambda^-_C \not\geq \Lambda(A_D)$. Together the two cases prove the contrapositive formulation of the assertion 
	\begin{equation*}
		\Lambda^-_C \geq \Lambda(A_D) \geq \Lambda^+_C \quad \text{\emph{implies}} \quad D \subseteq C.
	\end{equation*}
	
	\subsection*{Case 1}  Using Theorem~\ref{thm:locpot} and Proposition~\ref{prop:simultloc} (as $C\subset\Omega\setminus U$), there is a sequence $(f_j)$ in $L^2_\diamond(\Gamma)$ such that $u_j = u_{f_j}^{A^+_{C}}$ satisfy
	\begin{equation*}
		\lim_{j\to\infty} \int_B \abs{\nabla u_j}^2\,\di x = \infty \qquad \text{and} \qquad \lim_{j\to\infty}\int_{\Omega\setminus U} \abs{\nabla u_j}^2\,\di x = 0.
	\end{equation*}
	Note that $A_C^+ = A_0$ in $U$ where $A_D-A_0$ is positive semidefinite, and also in $B$ where $A_D-A_0$ is uniformly positive definite. From Proposition~\ref{prop:bnd}(ii) then $A_0A_D^{-1}(A_0-A_D)$ is negative semidefinite in $U$, and there exists a scalar $\tilde{c} > 0$ such that
	\begin{equation*}
		A_0A_D^{-1}(A_0-A_D) \leq -\tilde{c}I \text{ in } B.
	\end{equation*}
	Now Proposition~\ref{prop:mono} gives
	\begin{align*}
		\inner{f_j,\bigl[\Lambda(A_D) - \Lambda^+_C\bigr]f_j} &\leq \int_\Omega A^+_{C}A_D^{-1}(A^+_{C}-A_D)\nabla u_j\cdot\overline{\nabla u_j}\,\di x \\
		&\leq \norm{A^+_{C}A_D^{-1}(A^+_{C}-A_D)}_*\int_{\Omega\setminus U} \abs{\nabla u_j}^2\,\di x - \tilde{c}\int_B \abs{\nabla u_j}^2\,\di x.
	\end{align*}
	We conclude that
	\begin{equation*}
		\lim_{j\to\infty}\inner{f_j,\bigl[\Lambda(A_D) - \Lambda^+_C \bigr]f_j} = -\infty,
	\end{equation*}
	in particular, $\Lambda(A_D) \not\geq \Lambda^+_C$.
	
	\subsection*{Case 2} Using Theorem~\ref{thm:locpot} and Proposition~\ref{prop:simultloc} (as $C\subset\Omega\setminus U$), there is a sequence $(f_j)$ in $L^2_\diamond(\Gamma)$ such that $u_j = u_{f_j}^{A^-_{C}}$ satisfy
	\begin{equation*}
		\lim_{j\to\infty} \int_B \abs{\nabla u_j}^2\,\di x = \infty \qquad \text{and} \qquad \lim_{j\to\infty}\int_{\Omega\setminus U} \abs{\nabla u_j}^2\,\di x = 0.
	\end{equation*}
	Note that $A_C^- = A_0$ in $U$ where $A_D-A_0$ is negative semidefinite, and also in $B$ where $A_D-A_0$ is uniformly negative definite. From Proposition~\ref{prop:mono} we have
	\begin{align*}
		\inner{f_j,\bigl[\Lambda^-_C - \Lambda(A_D)\bigr]f_j} &\leq \int_\Omega (A_D-A^-_{C})\nabla u_j\cdot\overline{\nabla u_j}\,\di x \\
		&\leq \norm{A_D-A^-_{C}}_*\int_{\Omega\setminus U} \abs{\nabla u_j}^2\,\di x - c\int_B \abs{\nabla u_j}^2\,\di x
	\end{align*}
	for a scalar $c>0$. We conclude that
	\begin{equation*}
		\lim_{j\to\infty}\inner{f_j,\bigl[\Lambda^-_C - \Lambda(A_D)\bigr]f_j} = -\infty,
	\end{equation*}
	in particular, $\Lambda^-_C \not\geq \Lambda(A_D)$.
	
	\section{Proof of Theorem~\ref{thm:mainlinear}} \label{sec:mainlinearproof}
	
	Assume $D\subseteq C$ and let $u_0 = u_f^{A_0}$. From Proposition~\ref{prop:mono} we have
	\begin{align*}
		\inner{f,\bigl[\Lambda(A_D)-\Lambda(A_0)-\DLambda^+_{C}\bigr]f} &\geq \int_\Omega (A_0 - A_D)\nabla u_0\cdot\overline{\nabla u_0}\,\di x + \int_C(\beta I-A_0)\nabla u_0\cdot\overline{\nabla u_0}\,\di x \\
		&= \int_D (A_0 - A_D)\nabla u_0\cdot\overline{\nabla u_0}\,\di x + \int_C(\beta I-A_0)\nabla u_0\cdot\overline{\nabla u_0}\,\di x \\
		&\geq \int_{C\setminus D}(\beta I-A_0)\nabla u_0\cdot\overline{\nabla u_0}\,\di x \geq 0.
	\end{align*}
	
	For the next part we need an additional bound. Let $\xi\in\C^d$, then we have
	\begin{equation} \label{eq:matbnd}
		A_0A_D^{-1}A_0\xi\cdot\overline{\xi} = A_D^{-1}A_0\xi\cdot\overline{A_0\xi} \leq \alpha^{-1}\abs{A_0\xi}^2\leq \tfrac{\beta^2}{\alpha}\abs{\xi}^2.
	\end{equation}
	Note that $A_0 \leq \beta I \leq \frac{\beta^2}{\alpha}I$. Now from Proposition~\ref{prop:mono} and \eqref{eq:matbnd} we get
	\begin{align*}
		\inner{f,\bigl[\Lambda(A_0)-\Lambda(A_D)+\DLambda^-_{C}\bigr]f} & \\
		&\hspace{-3cm}\geq \int_\Omega A_0A_D^{-1}(A_D-A_0)\nabla u_0\cdot\overline{\nabla u_0}\,\di x + \int_C \bigl(\tfrac{\beta^2}{\alpha}I - A_0\bigr)\nabla u_0\cdot\overline{\nabla u_0}\,\di x \\
		&\hspace{-3cm}= \int_D (A_0 - A_0A_D^{-1}A_0)\nabla u_0\cdot\overline{\nabla u_0}\,\di x + \int_C \bigl(\tfrac{\beta^2}{\alpha}I - A_0\bigr)\nabla u_0\cdot\overline{\nabla u_0}\,\di x \\
		&\hspace{-3cm}\geq \int_{C\setminus D} \bigl(\tfrac{\beta^2}{\alpha}I - A_0\bigr)\nabla u_0\cdot\overline{\nabla u_0}\,\di x\geq 0.
	\end{align*}
	This concludes the proof for the first part of the theorem statement.
	
	Now assume that $D \not\subseteq C$, i.e.~that $D^\bullet\not\subseteq C$. Since both sets are closures of opens sets and have connected complements, there is a relatively open set that connects $D^\bullet\setminus C$ with $\Gamma$ that we may use for localization. Specifically, there exist a relatively open connected set $U\subset \overline{\Omega}\setminus C$ that intersects $\Gamma$ and an open ball $B\subset U \cap D$. Moreover, Assumption~\ref{assump:recon}(iii) ensures that we can pick $B$ and $U$ such that we arrive at one of two cases:
	\begin{itemize}
		\item Case 1: $A_D - A_0$ is positive semidefinite in $U$ and uniformly positive definite in $B$.
		\item Case 2: $A_D - A_0$ is negative semidefinite in $U$ and uniformly negative definite in $B$.
	\end{itemize}
	In the first case we will prove that $\Lambda(A_D) - \Lambda(A_0) \not\geq \DLambda^+_{C}$, and in the second case we will prove that $\DLambda^-_{C} \not\geq \Lambda(A_D) - \Lambda(A_0)$. Together the two cases prove the contrapositive formulation of the assertion 
	\begin{equation*}
		\DLambda^-_{C} \geq \Lambda(A_D) - \Lambda(A_0) \geq \DLambda^+_{C} \quad \text{\emph{implies}} \quad D \subseteq C.
	\end{equation*}
	For both cases we use the same sequence of localized potentials. From Theorem~\ref{thm:locpot} there is a sequence $(f_j)$ in $L^2_\diamond(\Gamma)$ such that $u_j = u_{f_j}^{A_0}$ satisfy
	\begin{equation*}
		\lim_{j\to\infty} \int_B \abs{\nabla u_j}^2\,\di x = \infty \qquad \text{and} \qquad \lim_{j\to\infty}\int_{\Omega\setminus U} \abs{\nabla u_j}^2\,\di x = 0.
	\end{equation*}
	
	\subsection*{Case 1} We now combine Proposition~\ref{prop:mono} with Proposition~\ref{prop:bnd}(ii), using the facts that $A_D-A_0$ is positive semidefinite in $U$ and $A_D-A_0\geq cI$ in $B$ for some scalar $c>0$, which results in
	\begin{align*}
		\inner{f_j,\bigl[\Lambda(A_D)-\Lambda(A_0)-\DLambda^+_{C}\bigr]f_j} &\\
		&\hspace{-4cm}\leq \int_\Omega A_0A_D^{-1}(A_0-A_D)\nabla u_j\cdot\overline{\nabla u_j}\,\di x + \int_C (\beta I - A_0)\nabla u_j\cdot\overline{\nabla u_j}\,\di x \\
		&\hspace{-4cm}\leq \norm{A_0A_D^{-1}(A_0-A_D)}_*\int_{\Omega\setminus U}\abs{\nabla u_j}^2\,\di x - c\bigl(\tfrac{\alpha}{\beta}\bigr)^2\int_B \abs{\nabla u_j}^2\,\di x + \norm{\beta I-A_0}_*\int_C \abs{\nabla u_j}^2\,\di x.
	\end{align*}
	Since $C\subset\Omega\setminus U$, we conclude that
	\begin{equation*}
		\lim_{j\to\infty}\inner{f_j,\bigl[\Lambda(A_D)-\Lambda(A_0)-\DLambda^+_{C}\bigr]f_j} = -\infty,
	\end{equation*}
	in particular, $\Lambda(A_D)-\Lambda(A_0)\not\geq \DLambda^+_{C}$.
	
	\subsection*{Case 2} We use Proposition~\ref{prop:mono} together with the facts that $A_D-A_0$ is negative semidefinite in $U$ and $A_D-A_0\leq -cI$ in $B$ for some scalar $c>0$, which results in
	\begin{align*}
		\inner{f_j,\bigl[\Lambda(A_0)-\Lambda(A_D)+\DLambda^-_{C}\bigr]f_j} &\\
		&\hspace{-4cm}\leq \int_\Omega (A_D-A_0)\nabla u_j\cdot\overline{\nabla u_j}\,\di x + \int_C \bigl(\tfrac{\beta^2}{\alpha} I - A_0\bigr)\nabla u_j\cdot\overline{\nabla u_j}\,\di x \\
		&\hspace{-4cm}\leq \norm{A_D - A_0}_*\int_{\Omega\setminus U}\abs{\nabla u_j}^2\,\di x - c\int_B \abs{\nabla u_j}^2\,\di x + \norm{\tfrac{\beta^2}{\alpha} I - A_0}_*\int_C \abs{\nabla u_j}^2\,\di x.
	\end{align*}
	Since $C\subset\Omega\setminus U$, we conclude that
	\begin{equation*}
		\lim_{j\to\infty}\inner{f_j,\bigl[\Lambda(A_0)-\Lambda(A_D)+\DLambda^-_{C}\bigr]f_j} = -\infty,
	\end{equation*}
	in particular, $\DLambda^-_{C} \not\geq \Lambda(A_D)-\Lambda(A_0)$.
	
	\section{Proof of Theorem~\ref{thm:mainlinearinnerpos}} \label{sec:mainposproof}
	
	Assume $B\subseteq D$ and let $u_0 = u_f^{A_0}$. By assumption, $A_D-A_0\geq cI$ in $D$ for some $c>0$, from which Proposition~\ref{prop:bnd}(ii) implies that
	\begin{equation*}
		A_0A_D^{-1}(A_D-A_0) \geq c\bigl(\tfrac{\alpha}{\beta}\bigr)^2 I \text{ in } D.
	\end{equation*}
	By Proposition~\ref{prop:mono}, we have
	\begin{align*}
		\inner{f,\bigl[ \DLambdah^+_B + \Lambda(A_0) - \Lambda(A_D) \bigr]f} &\geq -c\bigl(\tfrac{\alpha}{\beta}\bigr)^2\int_B \abs{\nabla u_0}^2\,\di x + \int_\Omega A_0A_D^{-1}(A_D-A_0)\nabla u_0\cdot\overline{\nabla u_0}\,\di x \\
		&\geq c\bigl(\tfrac{\alpha}{\beta}\bigr)^2\int_{D\setminus B} \abs{\nabla u_0}^2\,\di x \geq 0.
	\end{align*}
	
	Next, assume $B \not\subset D^\bullet$. Since $D^\bullet$ is closed and has connected complement, while $B$ is a non-empty open set, there exist a relatively open and connected set $U \subseteq \overline{\Omega}\setminus D^\bullet$ intersecting $\Gamma$ and~$B$, and an open ball $\widehat{B}\subset U\cap B$. From Theorem~\ref{thm:locpot} there is a sequence $(f_j)$ in $L^2_\diamond(\Gamma)$ such that $u_j = u_{f_j}^{A_0}$ satisfy
	\begin{equation*}
		\lim_{j\to\infty} \int_{\widehat{B}} \abs{\nabla u_j}^2\,\di x = \infty \qquad \text{and} \qquad \lim_{j\to\infty}\int_{\Omega\setminus U} \abs{\nabla u_j}^2\,\di x = 0.
	\end{equation*}
	Since $D\subset \Omega\setminus U$, Proposition~\ref{prop:mono} gives
	\begin{align*}
		\inner{f_j,\bigl[ \DLambdah^+_B + \Lambda(A_0) - \Lambda(A_D) \bigr]f_j} &\leq -c\bigl(\tfrac{\alpha}{\beta}\bigr)^2\int_B \abs{\nabla u_j}^2\,\di x + \int_\Omega (A_D-A_0)\nabla u_j\cdot\overline{\nabla u_j}\,\di x \\
		&\leq -c\bigl(\tfrac{\alpha}{\beta}\bigr)^2\int_{\widehat{B}} \abs{\nabla u_j}^2\,\di x + \norm{A_D-A_0}_*\int_{\Omega\setminus U} \abs{\nabla u_j}^2\,\di x.
	\end{align*}
	We conclude that
	\begin{equation*}
		\lim_{j\to\infty}\inner{f_j,\bigl[ \DLambdah^+_B + \Lambda(A_0) - \Lambda(A_D) \bigr]f_j} = -\infty,
	\end{equation*}
	in particular, $\Lambda(A_D) - \Lambda(A_0) \not\leq \DLambdah^+_{B}$. This proves the contrapositive formulation of the assertion
	\begin{equation*}
		\Lambda(A_D) - \Lambda(A_0) \leq \DLambdah^+_{B} \quad \text{\emph{implies}} \quad B \subset D^\bullet. 
	\end{equation*}
	
	\section{Proof of Theorem~\ref{thm:mainlinearinnerneg}} \label{sec:mainnegproof}
	
	Assume $B\subseteq D$ and let $u_0 = u_f^{A_0}$. By assumption, $A_0 - A_D \geq cI$ in $D$, so Proposition~\ref{prop:mono} yields
	\begin{align*}
		\inner{f,\bigl[ \Lambda(A_D) - \Lambda(A_0) - \DLambdah^-_{B} \bigr]f} &\geq -c\int_B \abs{\nabla u_0}^2\,\di x + \int_\Omega (A_0-A_D)\nabla u_0\cdot\overline{\nabla u_0}\,\di x \\
		&\geq c\int_{D\setminus B} \abs{\nabla u_0}^2\,\di x \geq 0.
	\end{align*}
	
	Next, assume $B \not\subset D^\bullet$. Since $D^\bullet$ is closed and has connected complement, while $B$ is a non-empty open set, there exist a relatively open and connected set $U \subseteq \overline{\Omega}\setminus D^\bullet$ intersecting $\Gamma$ and~$B$, and an open ball $\widehat{B}\subset U\cap B$. From Theorem~\ref{thm:locpot} there is a sequence $(f_j)$ in $L^2_\diamond(\Gamma)$ such that $u_j = u_{f_j}^{A_0}$ satisfy
	\begin{equation*}
		\lim_{j\to\infty} \int_{\widehat{B}} \abs{\nabla u_j}^2\,\di x = \infty \qquad \text{and} \qquad \lim_{j\to\infty}\int_{\Omega\setminus U} \abs{\nabla u_j}^2\,\di x = 0.
	\end{equation*}
	Since $D\subset \Omega\setminus U$, Proposition~\ref{prop:mono} gives
	\begin{align*}
		\inner{f_j,\bigl[ \Lambda(A_D) - \Lambda(A_0) - \DLambdah^-_{B} \bigr]f_j} &\leq -c\int_B \abs{\nabla u_j}^2\,\di x + \int_\Omega A_0A_D^{-1}(A_0-A_D)\nabla u_j\cdot\overline{\nabla u_j}\,\di x \\
		&\leq -c\int_{\widehat{B}} \abs{\nabla u_j}^2\,\di x + \norm{A_0A_D^{-1}(A_0-A_D)}_* \int_{\Omega\setminus U} \abs{\nabla u_j}^2\,\di x.
	\end{align*}
	We conclude that
	\begin{equation*}
		\lim_{j\to\infty}\inner{f_j,\bigl[ \Lambda(A_D) - \Lambda(A_0) - \DLambdah^-_{B} \bigr]f_j} = -\infty,
	\end{equation*}
	in particular, $\DLambdah^-_{B} \not\leq \Lambda(A_D) - \Lambda(A_0)$. This proves the contrapositive formulation of the assertion
	\begin{equation*}
		\DLambdah^-_{B} \leq \Lambda(A_D) - \Lambda(A_0) \quad \text{\emph{implies}} \quad B \subset D^\bullet. \qedhere
	\end{equation*}
	
	\section{The forward problem with extreme coefficients} \label{sec:forwardextreme}
	
	We will now consider extreme inclusions, meaning that part of the domain $C_0$ is occupied by a perfectly insulating conductivity (vanishing conductivity) and part of the domain $C_\infty$ is occupied by a perfectly conducting conductivity (infinite conductivity). The extreme parts are inherently isotropic. We may formally denote such a coefficient as
	\begin{equation} \label{eq:extremecoeff}
		A = \begin{dcases}
			A_\textup{F} & \text{in } \Omega\setminus C \\
			0I & \text{in } C_0 \\
			\infty I & \text{in } C_\infty,
		\end{dcases}
	\end{equation}
	where $C = C_0 \cup C_\infty$ and where $A_\textup{F}\in\mathcal{H}(\Omega)$ refers to the finite/non-extreme part of the coefficient. For notational convenience, since we later define extensions into $C_0$ and an inner product related to $\Omega\setminus C_0$, we will always consider $A_\textup{F}$ to be defined in all of $\Omega$. We apply the following assumption in connection with such extreme inclusions.
	\begin{assumption} \label{assump:Csets}
		Let $C_0,C_\infty\Subset\Omega$ satisfy:
		\begin{itemize}
			\item $C_0$ and $C_\infty$ are closures of open sets with Lipschitz boundaries.
			\item $C_0\cap C_\infty = \emptyset$.
			\item $\Omega\setminus C_0$ is connected.
		\end{itemize}
	\end{assumption}
	It is worth noting that $C_0$ and $C_\infty$ are allowed to comprise several connected components, and are also allowed to be empty sets. These extreme parts can be introduced via certain boundary conditions:
	\begin{align*}
		-\nabla \cdot(A_\textup{F}\nabla u) &= 0 \text{ in } \Omega\setminus C \\
		\nu\cdot (A_\textup{F}\nabla u) &= \begin{dcases}
			f & \text{on } \Gamma \\
			0 & \text{on } \partial(\Omega\setminus C_0)\setminus \Gamma
		\end{dcases} \\
		\nabla u &= 0 \text{ in } C_\infty^\circ \\
		\int_{\partial C_i} \nu\cdot(A_\textup{F}\nabla u)\,\di S &= 0 \text{ for each component $C_i$ of $C_\infty$}.
	\end{align*}
	In the last condition the trace is taken from the exterior of $C_\infty$. It is worth noting here, that $u$ is only defined in $\Omega\setminus C_0$ from this PDE problem. A consistent choice of extending $u$ into $C_0$, which we will use throughout this paper, is the $H^1$-extension $Eu$ of $u$ defined below.
	\begin{definition} \label{def:extension}
		Define $E \colon H^1(\Omega\setminus C_0) \to H^1(\Omega)$ as 
		\begin{equation*}
			Ew = \begin{dcases}
				w &\text{in } \Omega\setminus C_0 \\
				\widetilde{w} &\text{in } C_0,
			\end{dcases}
		\end{equation*} 
		with $\widetilde{w}$ being the unique solution in $H^1(C_0^\circ)$ of the Dirichlet problem
		\begin{align*}
			-\nabla \cdot(A_\textup{F}\nabla \widetilde{w}) &= 0 \text{ in } C_0^\circ \\
			\widetilde{w} &= w \text{ on } \partial C_0.
		\end{align*}
	\end{definition}
	It is clear that the extension $E$ depends on $C_0$ and $A_\textup{F}$, however to avoid introducing excessive notation, whenever we write $Eu$ it will be implicitly understood that we use the extension related to the $C_0$ and $A_\textup{F}$ from the definition of the specific electric potential $u$.
	
	Next we consider the weak formulation in connection with extreme inclusions, and for this purpose we introduce the space
	\begin{equation*}
		\mathscr{H}_{C_0}^{C_\infty} = \{\, v\in H_\diamond^1(\Omega\setminus C_0) \mid \nabla v = 0 \text{ in } C_\infty^\circ \,\}.
	\end{equation*}
	The weak problem now becomes
	\begin{equation} \label{eq:weakextreme}
		\int_{\Omega\setminus C_0} A_\textup{F}\nabla u\cdot\overline{\nabla v}\,\di x = \inner{f,v}
	\end{equation}
	for all $v\in \mathscr{H}_{C_0}^{C_\infty}$. Once again, the Lax--Milgram lemma ensures the existence of a unique solution $u = u_f^A\in\mathscr{H}_{C_0}^{C_\infty}$, and we have the local ND map $\Lambda(A) \colon f \mapsto u_f^A|_{\Gamma}$. 
	
	From \eqref{eq:weakextreme}, for any $f,g\in L^2_\diamond(\Gamma)$ and for any two pairs of extreme coefficients $A_1$ and $A_2$, with extreme inclusions $C = C_0\cup C_\infty$ for $A_1$ and where the perfectly conducting parts of $A_2$ are a subset of $C_\infty$, we have:
	\begin{equation} \label{eq:lambdaweakextreme}
		\inner{f,\Lambda(A_1)g} = \int_{\Omega\setminus C} A_2\nabla u_f^{A_2}\cdot \overline{\nabla Eu_g^{A_1}}\,\di x.
	\end{equation}
	In particular
	\begin{equation} \label{eq:lambdaweakextreme2}
		\inner{f,\Lambda(A)g} = \int_{\Omega\setminus C} A_\textup{F}\nabla u_f^{A}\cdot \overline{\nabla u_g^{A}}\,\di x = \int_{\Omega\setminus C_0} A_\textup{F}\nabla u_f^{A}\cdot \overline{\nabla u_g^{A}}\,\di x,
	\end{equation}
	for which $\Lambda(A)$ is self-adjoint if $A_\textup{F}\in \Hsa(\Omega)$.
	
	We moreover note that when $A_\textup{F}\in\Hsa(\Omega)$, then
	\begin{equation} \label{eq:innerA}
		\inner{w,v}_A = \int_{\Omega\setminus C_0} A_\textup{F}\nabla w \cdot\overline{\nabla v}\,\di x
	\end{equation}
	becomes an inner product on $\mathscr{H}_{C_0}^{\emptyset} \supseteq \mathscr{H}_{C_0}^{C_\infty}$, whose norm $\norm{\,\cdot\,}_A$ is equivalent with the norm on $H^1_\diamond(\Omega\setminus C_0)$. We will often make use of the norm equivalences between $\norm{\,\cdot\,}_{H^1(\Omega\setminus C_0)}$, $\norm{\nabla(\cdot)}_{L^2(\Omega\setminus C_0)^d}$, and $\norm{\,\cdot\,}_A$ that hold on $\mathscr{H}_{C_0}^{\emptyset}$.
	
	\section{Reconstruction of general anisotropic and extreme inclusions} \label{sec:mainextreme}
	
	We now give the main result on reconstructing anisotropic inclusions also including extreme inclusions. We moreover avoid requiring conductivity bounds such as those needed for Theorem~\ref{thm:main}. The main disadvantages, however, are twofold:
	\begin{itemize}
		\item We assume Lipschitz regularity of the outer inclusion boundaries, which is unlike in Theorems~\ref{thm:main} and~\ref{thm:mainlinear} where no explicit boundary regularity is needed for the inclusions.
		\item For numerical implementation, a linearized version of the method such as Theorem~\ref{thm:mainlinear} is very beneficial. For general inclusions this cannot be obtained in the extreme case. However, in the less general \emph{inner approach}, one can still obtain linearized versions of the method for either perfectly insulating \emph{or} perfectly conducting inclusions (Theorem~\ref{thm:extremelinear}).
	\end{itemize}
	We will assume to have a known (non-extreme) background conductivity $A_0\in \Hsa(\Omega)$ and an unknown conductivity $A_D$ of the form
	\begin{equation*}
		A_D = \begin{dcases}
			A_{\textup{F}} & \text{in } \Omega\setminus (D_0\cup D_\infty) \\
			0I & \text{in } D_0 \\
			\infty I & \text{in } D_\infty,
		\end{dcases}
	\end{equation*}
	where $D_0$ and $D_\infty$ satisfy Assumption~\ref{assump:Csets} (in place of $C_0$ and $C_\infty$), and with the finite part $A_\textup{F} \in \Hsa(\Omega)$. 
	
	Once again, we define the inclusions as
	\begin{equation*}
		D = \suppm(A_D - A_0),
	\end{equation*}
	and proceed to determine the outer shape $D^\bullet$. We denote by
	\begin{equation*}
		D_\textup{F} = D\setminus(D_0\cup D_\infty)
	\end{equation*}
	the part of the domain occupied by non-extreme deviations from $A_0$.
	
	We define the class of admissible test-inclusions, which in the extreme case also have Lipschitz regular boundary:
	\begin{align*}
		\widehat{\mathcal{A}} &= \{\, C \Subset \Omega \mid C \text{ is the closure of an open set,}  \\
		&\hphantom{{}= \{C \subset \Omega \mid{}\,}\text{has connected complement,} \\
		&\hphantom{{}= \{C \subset \Omega \mid{}\,}\text{and has Lipschitz boundary } \partial C \,\}.
	\end{align*}
	\begin{assumption}  \label{assump:reconextreme} \needspace{2\baselineskip}  {}\
		\begin{enumerate}[\rm(i)]
			\item Assume that $D\Subset \Omega$ and is the closure of an open set, and that $\partial D^\bullet$ is Lipschitz regular.
			\item For every $x\in\partial D^\bullet$ and every open neighborhood $W$ of $x$, assume there exists a relatively open connected set $V\subset D^\bullet\cap W$ that intersects $\partial D^\bullet$, satisfying either that $V\subseteq D_0$, $V\subseteq D_\infty$, or $V\subseteq D_\textup{F}$. In the latter case, $V$ can be picked such that either of two options hold:
			\begin{enumerate}[\rm(a)]
				\item $A_\textup{F} - A_0$ is positive semidefinite in $V$, and there exists an open ball $B\subset V$ on which $A_\textup{F} - A_0$ is uniformly positive definite.
				\item $A_\textup{F} - A_0$ is negative semidefinite in $V$, and there exists an open ball $B\subset V$ on which $A_\textup{F} - A_0$ is uniformly negative definite.
			\end{enumerate}
			\item Assume that $A_0$ satisfies the UCP.
		\end{enumerate}
	\end{assumption}
	
	For $C\in\widehat{\mathcal{A}}$, we define the test-coefficients
	\begin{equation*}
		A_C^{\emptyset} = \begin{dcases}
			0 I & \text{in } C \\
			A_0 & \text{in } \Omega\setminus C
		\end{dcases}
		\quad 
		\text{and}
		\quad
		A_{\emptyset}^C = \begin{dcases}
			\infty I & \text{in } C \\
			A_0 & \text{in } \Omega\setminus C,
		\end{dcases}
	\end{equation*}
	and the corresponding test-operators
	\begin{equation*}
		\Lambda_C^{\emptyset} = \Lambda(A_C^{\emptyset}) \quad \text{and}\quad \Lambda_{\emptyset}^C = \Lambda(A_{\emptyset}^C).
	\end{equation*}
	\begin{theorem} \label{thm:mainextreme}
		For any $C\in\widehat{\mathcal{A}}$ we have
		\begin{equation*}
			D \subseteq C \quad \text{implies} \quad \Lambda_C^{\emptyset} \geq \Lambda(A_D) \geq \Lambda_{\emptyset}^C.
		\end{equation*}
		Under Assumption~\ref{assump:reconextreme}, for any $C\in\widehat{\mathcal{A}}$ we have
		\begin{equation*}
			\Lambda_C^{\emptyset} \geq \Lambda(A_D) \geq \Lambda_{\emptyset}^C \quad \text{implies} \quad D \subseteq C.
		\end{equation*}
	\end{theorem}
	\begin{proof}
		The proof is given in Section~\ref{sec:mainproofextreme}.
	\end{proof}
	
	\begin{remark}
		Under Assumptions~\ref{assump:reconextreme}, Theorem~\ref{thm:mainextreme} gives
		\begin{equation*}
			D^\bullet = \cap\, \{\, C\in\widehat{\mathcal{A}} \mid \Lambda_C^{\emptyset} \geq \Lambda(A_D) \geq \Lambda_{\emptyset}^C \,\}.
		\end{equation*}
		Moreover, from the proof we also conclude:
		\begin{itemize}
			\item If $D_0 = \emptyset$ or is located away from $\partial D^\bullet$, and in Assumption~\ref{assump:reconextreme}(ii) there is only positive definiteness near $\partial D^\bullet$, we only need to check the inequality $\Lambda(A_D)\geq \Lambda_{\emptyset}^C$.
			\item If $D_\infty = \emptyset$ or is located away from $\partial D^\bullet$, and in Assumption~\ref{assump:reconextreme}(ii) there is only negative definiteness near $\partial D^\bullet$, we only need to check the inequality $\Lambda_C^{\emptyset} \geq \Lambda(A_D)$.
		\end{itemize} 
	\end{remark}
	
	 We also state the following two results for the inner approach of the monotonicity method, to \emph{either} determine perfectly insulating \emph{or} perfectly conducting inclusions. We do not give proofs of these results, as they are relatively short adaptations of the proofs of \cite[theorems 3.4 and 3.5]{Garde2020}, using the tools introduced in the present paper for the extreme anisotropic conductivities.
	 
	 To shorten the notation, we introduce the test-operators:
	 \begin{align*}
	 	\Lambda_{c,B} &= \Lambda(A_0 + c\chi_B I), \\
	 	\DLambda_{c,B} &= \DLambda(A_0;c\chi_B I),
	 \end{align*}
	 where $\chi_B$ is a characteristic function on the set $B$.
	
	\begin{theorem} \label{thm:extremelinear}
		Assume $D\in\widehat{\mathcal{A}}$, $A_0$ satisfies the UCP, and let $A_0 \geq c I$ in the Loewner order in $\Omega$ for some $c>0$. For any open set $B\subseteq \Omega$ we have
		\begin{align*}
			B \subset D \quad &\text{if and only if} \quad \DLambda_{-c,B}  \leq \Lambda_D^{\emptyset} - \Lambda(A_0) \\
			&\text{if and only if} \quad \Lambda_{\emptyset}^D - \Lambda(A_0) \leq \DLambda_{c,B}.
		\end{align*}
	\end{theorem}
	
	\begin{theorem} \label{thm:extreme2}
		Assume $D\in\widehat{\mathcal{A}}$, $A_0$ satisfies the UCP, and let $A_0 \geq \alpha I$ in the Loewner order in $\Omega$ for some $\alpha>0$. For any open set $B\subseteq \Omega$ we have
		\begin{align*}
			\text{If $0<c<\alpha$, then} \quad B \subset D \quad &\text{if and only if} \quad \Lambda_{-c,B} \leq \Lambda_{D}^{\emptyset}.  \\
			\text{If $c>0$, then} \quad B \subset D \quad &\text{if and only if} \quad \Lambda_{\emptyset}^D \leq \Lambda_{c,B}.
		\end{align*}
	\end{theorem}
	
	\section{Extreme measurements as a limit} \label{sec:extremelimit}
	
	Throughout this section, $C_0$ and $C_\infty$ satisfy Assumption~\ref{assump:Csets} and $C = C_0\cup C_\infty$. Moreover, $A$ is given by \eqref{eq:extremecoeff} with $A_\textup{F}\in \Hsa(\Omega)$.
	
	We define $P$ as the orthogonal projection of $\mathscr{H}_{C_0}^{\emptyset}$ onto $\mathscr{H}_{C_0}^{C_\infty}$, and define $P^\perp$ as the orthogonal projection of $\mathscr{H}_{C_0}^{\emptyset}$ onto $(\mathscr{H}_{C_0}^{C_\infty})^\perp$, both with respect to the inner product $\inner{\,\cdot\,,\,\cdot\,}_A$.
	\begin{lemma} \label{lemma:extremecharacterization}
		Let $f\in L^2_\diamond(\Gamma)$ and let $u_f^{A_\textup{F}}$ satisfy \eqref{eq:weak} with finite conductivity $A_\textup{F}$. Then the solution $u_f^A$ to \eqref{eq:weakextreme} is characterized as
		\begin{equation*}
			u_f^A = P(u_f^{A_\textup{F}}|_{\Omega\setminus C_0} - w),
		\end{equation*}
		for $w\in H_\diamond^1(\Omega\setminus C_0)$ solving
		\begin{align*}
			-\nabla\cdot(A_\textup{F}\nabla w) &= 0 \textup{ in } \Omega\setminus C_0 \\
			\nu\cdot(A_\textup{F}\nabla w) &= \nu\cdot(A_\textup{F}\nabla u_f^{A_\textup{F}}) \textup{ on } \partial C_0 \\
			\nu\cdot(A_\textup{F}\nabla w) &= 0 \textup{ on } \partial \Omega.
		\end{align*}
	\end{lemma}
	\begin{proof}
		It is evident that $\widehat{u} = u_f^{A_\textup{F}}|_{\Omega\setminus C_0} - w$ satisfies the corresponding conductivity equation in $\Omega\setminus C_0$, but with vanishing Neumann condition on $\partial C_0$ and Neumann condition $f$ on $\partial\Omega$. This corresponds to having the perfectly insulating inclusion in $C_0$, but still missing the perfectly conducting inclusion in $C_\infty$. 
		
		Since $P$ is self-adjoint, the weak form for $\widehat{u}$ implies
		\begin{eqnarray*}
			\inner{f,v} = \inner{\widehat{u},v}_A = \inner{\widehat{u},Pv}_A = \inner{P\widehat{u},v}_A
		\end{eqnarray*}
		for all $v\in\mathscr{H}_{C_0}^{C_\infty} \subseteq \mathscr{H}_{C_0}^{\emptyset}$. Hence $u_f^A = P\widehat{u}$ is the unique solution to the weak problem \eqref{eq:weakextreme}.
	\end{proof}
	
	\begin{lemma} \label{lemma:projest}
		Let $v\in\mathscr{H}_{C_0}^{\emptyset}$ satisfy $\nabla\cdot(A_\textup{F}\nabla v) = 0$ in $\Omega\setminus C$. There exists $K>0$ (independent of $v$) such that
		\begin{equation*}
			\norm{P^\perp v}_A \leq K\norm{\nabla v}_{L^2(C_\infty)^d}.
		\end{equation*}
	\end{lemma}
	\begin{proof}
		We let $\widehat{v}$ be the piecewise constant function in $C_\infty$, which in each component $C_i$ of $C_\infty$ equals the mean value of $v|_{C_i}$. Moreover, let $w\in\mathscr{H}_{C_0}^{C_\infty}$ solve the Dirichlet problem
		\begin{align*}
			-\nabla\cdot(A_\textup{F}\nabla w) &= 0 \text{ in } \Omega\setminus C \\
			w &= v \text{ on } \partial(\Omega\setminus C_0) \\
			w &= \widehat{v} \text{ in } C_\infty^\circ.
		\end{align*}
		Evidently $v-w$ satisfies a similar PDE problem in $\Omega\setminus C$ whose only non-trivial Dirichlet condition is on $\partial C_\infty$. We may now estimate $v-w$ using the continuous dependence on the Dirichlet condition, the trace theorem in $C_\infty^\circ$, and the Poincar\'e inequality applied in each component $C_i$:
		\begin{align*}
			\norm{v-w}_{H^1(\Omega\setminus C_0)}^2 &= \norm{v-w}_{H^1(\Omega\setminus C)}^2 + \norm{v-\widehat{v}}_{H^1(C_\infty^\circ)}^2 \\
			&\leq K\bigl(\norm{v-w}_{H^{1/2}(\partial C_\infty)}^2 + \norm{v-\widehat{v}}_{H^1(C_\infty^\circ)}^2\bigr) \\
			&\leq K\norm{v-\widehat{v}}_{H^1(C_\infty^\circ)}^2 \\
			&\leq K\norm{\nabla v}_{L^2(C_\infty)^d}^2.
		\end{align*}
		The proof is concluded using the projection theorem and norm equivalences:
		\begin{equation*}
			\norm{P^\perp v}_A = \inf_{\widetilde{v}\in\mathscr{H}_{C_0}^{C_\infty}}\norm{v-\tilde{v}}_A \leq \norm{v-w}_A \leq K\norm{\nabla v}_{L^2(C_\infty)^d}. \qedhere
		\end{equation*}
	\end{proof}
	
	These lemmas establish a convergence result for obtaining electric potentials and local ND maps for extreme coefficients as limits, based on a sequence of truncated coefficients. Recall Definition~\ref{def:extension} for the extension operator $E$.
	
	\begin{theorem} \label{thm:conv} \needspace{4\baselineskip} {}\
		Let $\epsilon > 0$ and consider the truncated conductivity coefficient
		\begin{equation*}
			A_\epsilon = \begin{dcases}
				A_\textup{F} & \text{in } \Omega\setminus C \\
				\epsilon A_\textup{F} & \text{in } C_0 \\
				\epsilon^{-1} A_\textup{F} & \text{in } C_\infty.
			\end{dcases}
		\end{equation*}
		There exists $K>0$ (only depending on $\Omega$, $C_0$, $C_\infty$, and bounds on $A_\textup{F}$) such that
		\begin{equation*}
			\norm{Eu_f^A - u_f^{A_\epsilon}}_{H^1(\Omega)} \leq K\epsilon^{1/2}\norm{f},
		\end{equation*}
		and consequently
		\begin{equation*}
			\norm{\Lambda(A) - \Lambda(A_\epsilon)}_{\mathscr{L}(L^2_\diamond(\Gamma))} \leq K\epsilon^{1/2}.
		\end{equation*}
	\end{theorem}
	
	\begin{remark} \label{remark:conv}
		It is straightforward to modify the proof of Theorem~\ref{thm:conv} to obtain the same estimates in case the coefficient is only partially truncated, i.e.\ for
		\begin{equation*}
			A_\epsilon = \begin{dcases}
				A_\textup{F} & \text{in } \Omega\setminus C \\
				\epsilon A_\textup{F} & \text{in } C_0 \\
				\infty I & \text{in } C_\infty
			\end{dcases}
			\quad \text{or} \quad 
			A_\epsilon = \begin{dcases}
				A_\textup{F} & \text{in } \Omega\setminus C \\
				0I & \text{in } C_0 \\
				\epsilon^{-1} A_\textup{F} & \text{in } C_\infty.
			\end{dcases}
		\end{equation*} 
	\end{remark}
	
	\begin{proof}
		To ease the presentation let $u_\epsilon = u_f^{A_\epsilon}$ and $u = u_f^A$, and let $c>0$ such that $A_\textup{F} \geq cI$ in the Loewner order in $\Omega$. We estimate $u_\epsilon$ in $C_0$ and $C_\infty$ by using the ``Cauchy inequality with $\epsilon$'' (here with $\alpha>0$ taking the role of the $\epsilon$ in the inequality) combined with the trace theorem in $\Omega\setminus C$:
		\begin{align*}
			c\bigl( \norm{\nabla u_\epsilon}_{L^2(\Omega\setminus C)}^2 + \epsilon\norm{\nabla u_\epsilon}_{L^2(C_0)}^2 + \epsilon^{-1}\norm{\nabla u_\epsilon}_{L^2(C_\infty)}^2 \bigr) &\leq \int_\Omega A_\epsilon\nabla u_\epsilon\cdot\overline{\nabla u_\epsilon}\,\di x \\
			&= \inner{f,u_\epsilon} \\
			&\leq \alpha\norm{f}^2+\frac{1}{4\alpha}\norm{u_\epsilon}_{H^{1/2}(\partial(\Omega\setminus C))}^2 \\
			&\leq \alpha\norm{f}^2 + \frac{\widehat{K}}{4\alpha}\norm{\nabla u_\epsilon}_{L^2(\Omega\setminus C)^d}^2.
		\end{align*} 
		Choosing $\alpha = \frac{\widehat{K}}{4c}$ results in
		\begin{equation*}
			\epsilon\norm{\nabla u_\epsilon}_{L^2(C_0)^d}^2 + \epsilon^{-1}\norm{\nabla u_\epsilon}_{L^2(C_\infty)^d}^2 \leq \frac{\widehat{K}}{4c^2}\norm{f}^2,
		\end{equation*}
		or rather there is a constant $K>0$ such that
		\begin{align}
			\norm{\nabla u_\epsilon}_{L^2(C_\infty)^d} &\leq K\epsilon^{1/2}\norm{f}, \label{eq:uepsbnd1}\\
			\norm{\nabla u_\epsilon}_{L^2(C_0)^d} & \leq K\epsilon^{-1/2}\norm{f}. \label{eq:uepsbnd2}
		\end{align}
		We will make use of these bounds momentarily. First we consider $u_\epsilon|_{\Omega\setminus C_0} \in \mathscr{H}_{C_0}^{\emptyset}$ and apply the self-adjointness of the projections $P$ and $P^\perp$, and the Cauchy--Schwarz inequality, to estimate $u-u_\epsilon$:
		\begin{align}
			\norm{u-u_\epsilon}_A^2 &= \norm{P(u-u_\epsilon)}_A^2 + \norm{P^\perp(u-u_\epsilon)}_A^2 \notag \\
			&\leq \inner{u-u_\epsilon,P(u-u_\epsilon)}_A + \norm{P^\perp u_\epsilon}_A\norm{u-u_\epsilon}_A. \label{eq:convest}
		\end{align}
		A combination of Lemma~\ref{lemma:projest} and \eqref{eq:uepsbnd1} gives the estimate
		\begin{equation} \label{eq:convterm2}
			\norm{P^\perp u_\epsilon}_A\norm{u-u_\epsilon}_A \leq K\epsilon^{1/2}\norm{f}\norm{u-u_\epsilon}_A.
		\end{equation}
		What remains is to bound the first term in \eqref{eq:convest} in a similar manner.
		
		Since $P(u-u_\epsilon)$ has vanishing gradient in $C_\infty^\circ$, the weak form for $u_\epsilon$ gives
		\begin{equation*}
			\inner{u_\epsilon,P(u-u_\epsilon)}_A = \inner{f,P(u-u_\epsilon)} - \epsilon\int_{C_0} A_\textup{F}\nabla u_\epsilon\cdot\overline{\nabla w}\,\di x,
		\end{equation*}
		where $w$ is an $H^1$-extension of $P(u-u_\epsilon)$ onto $C_0$. Hence from the weak form of $u$, we have
		\begin{equation} \label{eq:convtmp1}
			\inner{u-u_\epsilon,P(u-u_\epsilon)}_A = \epsilon\int_{C_0} A_\textup{F}\nabla u_\epsilon\cdot\overline{\nabla w}\,\di x.
		\end{equation}
		The boundedness of the $H^1$-extension operator and the projection $P$, together with norm equivalences, imply
		\begin{equation*}
			\norm{\nabla w}_{L^2(C_0)^d} \leq \norm{w}_{H^1(\Omega)} \leq K\norm{u-u_\epsilon}_A.
		\end{equation*}
		Combining this with \eqref{eq:convtmp1} and \eqref{eq:uepsbnd2} leads to the estimate
		\begin{equation} \label{eq:convtmp2}
			\inner{u-u_\epsilon,P(u-u_\epsilon)}_A \leq K\epsilon^{1/2}\norm{f}\norm{u-u_\epsilon}_A.
		\end{equation}
		
		Finally, \eqref{eq:convest}, \eqref{eq:convterm2}, and \eqref{eq:convtmp2}, together with norm equivalences, results in
		\begin{equation} \label{eq:convres1}
			\norm{u-u_\epsilon}_{H^1(\Omega\setminus C_0)} \leq K\epsilon^{1/2}\norm{f}.
		\end{equation} 
		We now consider the extension $Eu$ and note that $v = (Eu-u_\epsilon)|_{C_0}$ is a weak solution of
		\begin{align*}
			-\nabla\cdot(A_\textup{F}\nabla v) &= 0 \text{ in } C_0^\circ \\
			v &= u-u_\epsilon \text{ on } \partial C_0.
		\end{align*}
		From the continuous dependence on the Dirichlet condition, the trace theorem in $\Omega\setminus C_0$, and \eqref{eq:convres1} we have
		\begin{equation*}
			\norm{Eu-u_\epsilon}_{H^1(C_0^\circ)} \leq K\norm{u-u_\epsilon}_{H^{1/2}(\partial C_0)} \leq K\norm{u-u_\epsilon}_{H^1(\Omega\setminus C_0)} \leq K\epsilon^{1/2}\norm{f}. 
		\end{equation*}
		Combined with \eqref{eq:convres1}, this proves the first part of the theorem.
		
		The estimate for the local ND maps follows from \eqref{eq:convres1} and the trace theorem in $\Omega\setminus C_0$:
		\begin{equation*}
			\norm{\Lambda(A)-\Lambda(A_\epsilon)}_{\mathscr{L}(L_\diamond^2(\Gamma))} = \sup_{\norm{f}=1}\norm{u-u_\epsilon} \leq K \sup_{\norm{f}=1}\norm{u-u_\epsilon}_{H^1(\Omega\setminus C_0)} \leq K\epsilon^{1/2}. \qedhere
		\end{equation*}
	\end{proof}
	
	\section{Extreme operator inequalities} \label{sec:monoineqextreme}
	
	Throughout this section, $C_0$ and $C_\infty$ satisfy Assumption~\ref{assump:Csets} and $C = C_0\cup C_\infty$. We re-emphasize that $C_0$ and $C_\infty$ are allowed to be empty sets, since this will be particularly relevant when using the following operator inequalities in the proof of Theorem~\ref{thm:mainextreme} in Section~\ref{sec:mainproofextreme}. We also let $f\in L^2_\diamond(\Gamma)$ as usual.
	\begin{proposition} \label{prop:monoextreme1}
	Let $A_1,A_2\in\Hsa(\Omega)$ and for $j\in\{1,2\}$ define
	\begin{equation*}
		\widehat{A}_j = \begin{dcases}
			A_j &\text{in } \Omega\setminus C \\
			0I &\text{in } C_0 \\
			\infty I & \text{in } C_\infty.
		\end{dcases}
	\end{equation*}
	We denote $\Lambda_j = \Lambda(\widehat{A}_j)$ and $u_j = u_f^{\widehat{A}_j}$. Then
	\begin{align}
		\inner{f,(\Lambda_1-\Lambda_2)f} &\geq \int_{\Omega\setminus C} (A_2-A_1)\nabla u_2\cdot \overline{\nabla u_2}\,\di x, \label{eq:monoext11} \\
		\inner{f,(\Lambda_1-\Lambda_2)f} &\leq \int_{\Omega\setminus C} A_2 A_1^{-1}(A_2-A_1)\nabla u_2\cdot\overline{\nabla u_2}\,\di x. \label{eq:monoext12}
	\end{align}
	\end{proposition}
	\begin{proof}
		Using the weak forms, we obtain
		\begin{equation*}
			\inner{u_1,u_2}_{\widehat{A}_1} = \inner{f,\Lambda_2 f} = \norm{u_2}_{\widehat{A}_2}^2.
		\end{equation*}
		Therefore
		\begin{align}
			0 &\leq \norm{u_1-u_2}_{\widehat{A}_1}^2 \notag\\
			&= \norm{u_1}_{\widehat{A}_1}^2 + \norm{u_2}_{\widehat{A}_1}^2 - 2\redel\inner{u_1,u_2}_{\widehat{A}_1} \notag\\
			&= \inner{f,\Lambda_1 f} - \inner{f,\Lambda_2 f} + \norm{u_2}_{\widehat{A}_1}^2 - \norm{u_2}_{\widehat{A}_2}^2, \label{eq:monotmp1}
		\end{align}
		which is just a different way of stating \eqref{eq:monoext11}, since $u_j$ has vanishing gradient in $C_\infty^\circ$.
		
		Next we prove \eqref{eq:monoext12}:
		\begin{align*}
			&\int_{\Omega\setminus C} A_2 A_1^{-1}(A_2-A_1)\nabla u_2\cdot\overline{\nabla u_2}\,\di x \\
			&\hspace{1.5cm}\geq \int_{\Omega\setminus C} A_2 A_1^{-1}(A_2-A_1)\nabla u_2\cdot\overline{\nabla u_2}\,\di x - \int_{\Omega\setminus C} \abs{A_1^{1/2}\nabla u_1 - A_1^{-1/2}A_2\nabla u_2}^2\,\di x \\
			&\hspace{1.5cm}= - \int_{\Omega\setminus C} A_2\nabla u_2\cdot\overline{\nabla u_2}\,\di x -\int_{\Omega\setminus C} A_1\nabla u_1\cdot\overline{\nabla u_1}\,\di x + 2\redel\Bigl(\int_{\Omega\setminus C} A_2\nabla u_2\cdot\overline{\nabla u_1}\,\di x\Bigr)  \\
			&\hspace{1.5cm}= \inner{f,(\Lambda_1 - \Lambda_2)f}. \qedhere
		\end{align*}
	\end{proof}
	
	In the following two results we let $A$ be as in \eqref{eq:extremecoeff} with $A_\textup{F}\in\Hsa(\Omega)$ and with the extreme inclusion $C_0$ and $C_\infty$. We also introduce the short-hand notation $\Lambda_{C_0}^{C_\infty} = \Lambda(A)$ for the local ND map and $u_{C_0}^{C_\infty} = u_f^{A}$ for the electric potential.
	
	\begin{proposition} \label{prop:monoextreme2}
		There exists $K>0$ (independent of $f$) such that
		\begin{align}
			\inner{f,(\Lambda_{C_0}^{\emptyset} - \Lambda_{C_0}^{C_\infty})f} &\geq \int_{C_\infty}A_\textup{F}\nabla u_{C_0}^{\emptyset}\cdot\overline{\nabla u_{C_0}^{\emptyset}}\,\di x \label{eq:monoext21} \\
			\inner{f,(\Lambda_{C_0}^{\emptyset} - \Lambda_{C_0}^{C_\infty})f} &\leq K\int_{C_\infty}\abs{\nabla u_{C_0}^{\emptyset}}^2\,\di x. \label{eq:monoext22}
		\end{align}
	\end{proposition}
	\begin{proof}
		We start by proving \eqref{eq:monoext21}. We define a truncated coefficient (in $C_\infty$), for $\epsilon>0$:
		\begin{equation*}
			A_\epsilon = \begin{dcases}
				A_\textup{F} &\text{in } \Omega\setminus C \\
				0I & \text{in } C_0 \\
				\epsilon^{-1}A_\textup{F} &\text{in } C_\infty.
			\end{dcases}
		\end{equation*}
		Using \eqref{eq:monoext12} from Proposition~\ref{prop:monoextreme1} (here with $\Omega\setminus C$ replaced by $\Omega\setminus C_0$), we have
		\begin{equation*}
			\inner{f,\bigl(\Lambda_{C_0}^{\emptyset}-\Lambda(A_\epsilon)\bigr)f} \geq (1-\epsilon)\int_{C_\infty}A_\textup{F}\nabla u_{C_0}^\emptyset\cdot\overline{\nabla u_{C_0}^{\emptyset}}\,\di x.
		\end{equation*}
		We arrive at \eqref{eq:monoext21} by letting $\epsilon\to 0$ due to Theorem~\ref{thm:conv} (Remark~\ref{remark:conv}).
		
		Next we consider \eqref{eq:monoext22}. Note that the proof of Lemma~\ref{lemma:extremecharacterization} implies that $u_{C_0}^{C_\infty} = Pu_{C_0}^{\emptyset}$, and therefore $P^\perp u_{C_0}^{\emptyset} = u_{C_0}^{\emptyset} - u_{C_0}^{C_\infty}$. Inserting this into the weak form for $u_{C_0}^{\emptyset}$ gives
		\begin{equation*}
			\inner{f,(\Lambda_{C_0}^{\emptyset} - \Lambda_{C_0}^{C_\infty})f} = \inner{f,P^\perp u_{C_0}^{\emptyset}} = \inner{u_{C_0}^{\emptyset},P^\perp u_{C_0}^{\emptyset}}_A = \norm{P^\perp u_{C_0}^{\emptyset}}_A^2.
		\end{equation*}
		Hence, the result follows by applying Lemma~\ref{lemma:projest}.		
	\end{proof}
	
	\begin{proposition} \label{prop:monoextreme3}
		We have the bounds
		\begin{align}
			\inner{f,(\Lambda_{C_0}^{C_\infty} - \Lambda_{\emptyset}^{C_\infty})f} &\geq \int_{C_0}A_\textup{F}\nabla u_{\emptyset}^{C_\infty}\cdot\overline{\nabla u_{\emptyset}^{C_\infty}}\,\di x \label{eq:monoext31} \\
			\inner{f,(\Lambda_{C_0}^{C_\infty} - \Lambda_{\emptyset}^{C_\infty})f} &\leq \int_{C_0}A_\textup{F}\nabla Eu_{C_0}^{C_\infty}\cdot\overline{\nabla Eu_{C_0}^{C_\infty}}\,\di x. \label{eq:monoext32}
		\end{align}
	\end{proposition}
	\begin{proof}
		Once again we start by introducing a truncated coefficient (this time in $C_0$):
		\begin{equation*}
			A_\epsilon = \begin{dcases}
				A_\textup{F} &\text{in } \Omega\setminus C \\
				\epsilon A_\textup{F} & \text{in } C_0 \\
				\infty I &\text{in } C_\infty.
			\end{dcases}
		\end{equation*}
		Using \eqref{eq:monoext11} from Proposition~\ref{prop:monoextreme1} (here with $\Omega\setminus C$ replaced by $\Omega\setminus C_\infty$), we have
		\begin{equation*}
			\inner{f,\bigl(\Lambda(A_\epsilon) - \Lambda_{\emptyset}^{C_\infty}\bigr)f} \geq (1-\epsilon)\int_{C_0}A_\textup{F}\nabla u_{\emptyset}^{C_\infty}\cdot\overline{\nabla u_{\emptyset}^{C_\infty}}\,\di x.
		\end{equation*}
		We arrive at \eqref{eq:monoext31} by letting $\epsilon\to 0$ due to Theorem~\ref{thm:conv} (Remark~\ref{remark:conv}).
		
		Next we consider \eqref{eq:monoext32}. Let $\widehat{A}$ be the coefficient associated with $\Lambda_{\emptyset}^{C_\infty}$, recall the extension operator $E$ (Definition~\ref{def:extension}), and note that
		\begin{equation*}
			\inner{u_\emptyset^{C_\infty},Eu_{C_0}^{C_\infty}}_{\widehat{A}} = \inner{f,\Lambda_{C_0}^{C_\infty}f} = \norm{Eu_{C_0}^{C_\infty}}_{\widehat{A}}^2 - \int_{C_0}A_\textup{F}\nabla Eu_{C_0}^{C_\infty}\cdot\overline{\nabla Eu_{C_0}^{C_\infty}}\,\di x.
		\end{equation*}
		The inequality \eqref{eq:monoext32} is therefore given by the following computation:
		\begin{align*}
			0 &\leq \norm{Eu_{C_0}^{C_\infty}-u_{\emptyset}^{C_\infty}}_{\widehat{A}}^2 \\
			&= \norm{Eu_{C_0}^{C_\infty}}_{\widehat{A}}^2 + \norm{u_{\emptyset}^{C_\infty}}_{\widehat{A}}^2 - 2\redel\inner{u_{\emptyset}^{C_\infty},Eu_{C_0}^{C_\infty}}_{\widehat{A}} \\
			&= \inner{f,(\Lambda_{\emptyset}^{C_\infty} - \Lambda_{C_0}^{C_\infty})f} + \int_{C_0}A_\textup{F}\nabla Eu_{C_0}^{C_\infty}\cdot\overline{\nabla Eu_{C_0}^{C_\infty}}\,\di x. \qedhere
		\end{align*}
	\end{proof}
		
	\section{Localized potentials for extreme coefficients} \label{sec:locextreme}
	
	Throughout this section, $C_0$ and $C_\infty$ satisfy Assumption~\ref{assump:Csets} and $C = C_0\cup C_\infty$. Moreover, $A$ is given by \eqref{eq:extremecoeff} with $A_\textup{F}\in \Hsa(\Omega)$. 
	
	We have previously introduced localized potentials for anisotropic coefficients in Theorem~\ref{thm:locpot}, and how one can transfer the localization between different coefficients in Propostion~\ref{prop:simultloc}, provided the coefficients only differ in the set where the power tends to zero. Below we show that such localization can also be transferred to the case of extreme inclusions, again provided that the extreme inclusions are located in the set where the power tends to zero.
	\begin{proposition} \label{prop:locextreme}
		Let $B$ and $U$ be as in Theorem~\ref{thm:locpot}, and assume for $(f_j)$ in $L^2_\diamond(\Gamma)$ that $u_j = u_{f_j}^{A_\textup{F}}$ satisfy
		\begin{equation*}
			\lim_{j\to\infty} \int_B \abs{\nabla u_j}^2\,\di x = \infty \qquad \text{and} \qquad \lim_{j\to\infty}\int_{\Omega\setminus U} \abs{\nabla u_j}^2\,\di x = 0.
		\end{equation*}
		If $C \subset \Omega\setminus \overline{U}$ then $\widehat{u}_j = Eu_{f_j}^{A}$ also satisfy
		\begin{equation*}
			\lim_{j\to\infty} \int_B \abs{\nabla \widehat{u}_j}^2\,\di x = \infty \qquad \text{and} \qquad \lim_{j\to\infty}\int_{\Omega\setminus U} \abs{\nabla \widehat{u}_j}^2\,\di x = 0.
		\end{equation*}
	\end{proposition}
	\begin{proof}
		Let $v_j = u_j|_{\Omega\setminus C_0} - w_j$ for $w_j$ solving the auxiliary PDE problem from Lemma~\ref{lemma:extremecharacterization} (with $f$ replaced by $f_j$), and consider $\widehat{v}_j = u_{f_j}^A = Pv_j$.
		
		Let $V$ be an open set satisfying $C_0 \subset \Omega\setminus\overline{V} \subseteq \Omega\setminus\overline{U}$, and such that $\partial(\Omega\setminus\overline{V})$ is Lipschitz regular. By continuous dependence of $w_j$ on its Neumann condition and from the weak form of $u_j|_{\Omega\setminus(V\cup C_0)}$, we have the estimate
		\begin{align*}
			\norm{w_j}_{H^1(\Omega\setminus C_0)} &\leq K\norm{\nu\cdot(A_\textup{F}\nabla u_j)}_{H^{-1/2}(\partial C_0)} \\
			&\leq K\norm{\nu\cdot(A_\textup{F}\nabla u_j)}_{H^{-1/2}(\partial (\Omega\setminus(V\cup C_0)))} \\
			&\leq K\norm{\nabla u_j}_{L^2(\Omega\setminus(U\cup C_0))^d} \to 0 \text{ for } j\to\infty.
		\end{align*}
		Consequently, from the limits of $u_j$ and $w_j$, we have
		\begin{equation} \label{eq:conv_vj}
			\lim_{j\to\infty} \int_B \abs{\nabla v_j}^2\,\di x = \infty \qquad \text{and} \qquad \lim_{j\to\infty}\int_{\Omega\setminus (U\cup C_0)} \abs{\nabla v_j}^2\,\di x = 0,
		\end{equation}
		and from Lemma~\ref{lemma:projest} we have
		\begin{equation} \label{eq:conv_vj2}
			\norm{P^\perp v_j}_{A} \leq K\norm{\nabla v_j}_{L^2(C_\infty)^d} \to 0 \text{ for } j \to\infty.
		\end{equation}
		Since $\widehat{v}_j = v_j - P^\perp v_j$, then a combination of \eqref{eq:conv_vj} and \eqref{eq:conv_vj2} gives
		\begin{equation*} 
			\lim_{j\to\infty} \int_B \abs{\nabla \widehat{v}_j}^2\,\di x = \infty \qquad \text{and} \qquad \lim_{j\to\infty}\int_{\Omega\setminus (U\cup C_0)} \abs{\nabla \widehat{v}_j}^2\,\di x = 0.
		\end{equation*}
		We have $\widehat{u}_j = E\widehat{v}_j$, and have thus proven the limiting behavior in $\Omega\setminus C_0$. What remains is to investigate $\widehat{u}_j|_{C_0}$. 
		
		In the following we may assume that $\Omega\setminus(\overline{V}\cup C_0)$ is connected, otherwise a similar estimate can be achieved by working in each component separately that touches $C_0$. Let $c_j$ denote the mean of $\widehat{u}_j$ in $\Omega\setminus(\overline{V}\cup C_0)$, then from the trace theorem and Poincar\'e inequality we have
		\begin{align*}
			\norm{\widehat{u}_j-c_j}_{H^{1/2}(\partial C_0)} &\leq K\norm{\widehat{u}_j-c_j}_{H^1(\Omega\setminus(\overline{V}\cup C_0))} \\
			&\leq K\norm{\nabla\widehat{u}_j}_{L^2(\Omega\setminus(V\cup C_0))^d} \to 0 \text{ for } j\to\infty.
		\end{align*}
		However, as $\widehat{u}_j-c_j$ solves the same PDE as $\widehat{u}_j$ in $C_0$, we conclude from the continuous dependence on the Dirichlet condition that
		\begin{equation*}
			\norm{\nabla \widehat{u}_j}_{L^2(C_0)^d} = \norm{\nabla(\widehat{u}_j-c_j)}_{L^2(C_0)^d} \leq K\norm{\widehat{u}_j-c_j}_{H^{1/2}(\partial C_0)}\to 0 \text{ for } j\to\infty. \qedhere
		\end{equation*}
	\end{proof}
	
	\section{Proof of Theorem \ref{thm:mainextreme}} \label{sec:mainproofextreme}
	
	Recall that $A_0\in \Hsa(\Omega)$ and the unknown conductivity $A_D$ is of the form
	\begin{equation*}
		A_D = \begin{dcases}
			A_{\textup{F}} & \text{in } \Omega\setminus (D_0\cup D_\infty) \\
			0I & \text{in } D_0 \\
			\infty I & \text{in } D_\infty,
		\end{dcases}
	\end{equation*}
	where $D_0$ and $D_\infty$ satisfy Assumption~\ref{assump:Csets} (in place of $C_0$ and $C_\infty$), and with the finite part $A_\textup{F} \in \Hsa(\Omega)$. The inclusions are
	\begin{equation*}
		D = \suppm(A_D - A_0),
	\end{equation*}
	and recall that 
	\begin{equation*}
		D_\textup{F} = D\setminus(D_0\cup D_\infty)
	\end{equation*}
	is the part of the domain occupied by non-extreme deviations from $A_0$.
	
	\subsection*{Proof of ``$\boldsymbol{D\subseteq C \Rightarrow \Lambda_{C}^{\emptyset} \geq \Lambda(A_D) \geq \Lambda_{\emptyset}^{C}}$''}
	
	We truncate $A_D$ for an $\epsilon>0$ as
	\begin{equation*}
		A_\epsilon = \begin{dcases}
			A_{\textup{F}} & \text{in } \Omega\setminus (D_0\cup D_\infty) \\
			\epsilon A_\textup{F} & \text{in } D_0 \\
			\epsilon^{-1} A_\textup{F} & \text{in } D_\infty,
		\end{dcases}
	\end{equation*}
	and similarly truncate the test-coefficients
	\begin{equation*}
		\widehat{A}_\epsilon = \begin{dcases}
			A_\textup{F} & \text{in } \Omega\setminus C \\
			\epsilon A_\textup{F} & \text{in } C
		\end{dcases}
		\qquad \text{and} \qquad
		\widetilde{A}_\epsilon = \begin{dcases}
			A_\textup{F} & \text{in } \Omega\setminus C \\
			\epsilon^{-1} A_\textup{F} & \text{in } C.
		\end{dcases}
	\end{equation*}
	For $D\subseteq C$ and $\epsilon\in(0,1]$ we have $\widehat{A}_\epsilon \leq A_\epsilon \leq \widetilde{A}_\epsilon$, which by \eqref{eq:simplemono} implies $\Lambda(\widehat{A}_\epsilon) \geq \Lambda(A_\epsilon) \geq \Lambda(\widetilde{A}_\epsilon)$. Letting $\epsilon\to 0$, the operator norm convergence from Theorem~\ref{thm:conv} gives
	\begin{equation*}
		\Lambda_{C}^{\emptyset} \geq \Lambda(A_D) \geq \Lambda_{\emptyset}^{C}.
	\end{equation*}
	
	\subsection*{Proof of ``$\boldsymbol{\Lambda_{C}^{\emptyset} \geq \Lambda(A_D) \geq \Lambda_{\emptyset}^{C} \Rightarrow D\subseteq C}$''}
	
	Now assume that $D \not\subseteq C$, i.e.~that $D^\bullet\not\subseteq C$. Since both sets are closures of opens sets and have connected complements, there is a relatively open set that connects $D^\bullet\setminus C$ with $\Gamma$ that we may use for localization. Specifically, there exist a relatively open connected set $U\subset \overline{\Omega}\setminus C$ that intersects $\Gamma$ and an open ball $B\subset U \cap D$. Moreover, Assumption~\ref{assump:reconextreme}(ii) ensures that we can pick $B$ and $U$ such that we arrive at one of four cases:
	\begin{itemize}
		\item Case 1: $U\cap D \subseteq D_\infty$.
		\item Case 2: $U\cap D \subseteq D_\textup{F}$, and $A_\textup{F} - A_0$ is positive semidefinite in $U$ and uniformly positive definite in $B$.
		\item Case 3: $U\cap D \subseteq D_0$.
		\item Case 4: $U\cap D \subseteq D_\textup{F}$, and $A_\textup{F} - A_0$ is negative semidefinite in $U$ and uniformly negative definite in $B$.
	\end{itemize}
	In the first two cases we will prove that $\Lambda(A_D) \not\geq \Lambda_{\emptyset}^{C}$, and in the last two case we will prove that $\Lambda_{C}^{\emptyset} \not\geq \Lambda(A_D)$. Together the four cases prove the contrapositive formulation of the assertion 
	\begin{equation*}
		\Lambda_{C}^{\emptyset} \geq \Lambda(A_D) \geq \Lambda_{\emptyset}^{C} \quad \text{\emph{implies}} \quad D\subseteq C.
	\end{equation*}
	
	In the proofs of the four cases we will have several comparisons of coefficients, so we use the convention that
	\begin{equation*}
		A_\textup{F} = A_0 \text{ in } D_0 \cup D_\infty
	\end{equation*}
	to simplify the presentation (since the values of $A_\textup{F}$ in $D_0$ and $D_\infty$ are not specified from the definition of $A_D$). We will also make use of the bounds 
	\begin{equation*}
		\alpha I \leq A_0 \leq \beta I \quad \text{and} \quad \alpha I \leq A_\textup{F} \leq \beta I
	\end{equation*}
	in the Loewner order in $\Omega$, for suitable $\alpha,\beta>0$.
	
	\subsection*{Case 1}
	
	We define two auxiliary coefficients:
	\begin{equation*}
		A_1 = A_{\textup{F}}
		\qquad \text{and} \qquad 
		A_2 = \begin{dcases}
			A_{\textup{F}} & \text{in } \Omega\setminus D_0 \\
			0 I & \text{in } D_0.
		\end{dcases}
	\end{equation*}
	We shorten the notation to $\Lambda = \Lambda(A_D)$ and $\Lambda_k = \Lambda(A_k)$ for $k\in\{0,1,2\}$. From Theorem~\ref{thm:locpot} and Propositions~\ref{prop:simultloc} and~\ref{prop:locextreme}, we may pick a sequence $(f_j)$ in $L^2_\diamond(\Gamma)$, such that $u_j = u_{f_j}^{A_0}$, $\widehat{u}_j = u_{f_j}^{A_1}$, and $\widetilde{u}_j = Eu_{f_j}^{A_2}$ satisfy
	\begin{equation*}
		\lim_{j\to\infty} \int_B \abs{\nabla u_j}^2\,\di x = \lim_{j\to\infty} \int_B \abs{\nabla \widehat{u}_j}^2\,\di x = \lim_{j\to\infty} \int_B \abs{\nabla \widetilde{u}_j}^2\,\di x = \infty
	\end{equation*}
	and
	\begin{equation*}
		\lim_{j\to\infty} \int_{\Omega\setminus U} \abs{\nabla u_j}^2\,\di x = \lim_{j\to\infty} \int_{\Omega\setminus U} \abs{\nabla \widehat{u}_j}^2\,\di x = \lim_{j\to\infty} \int_{\Omega\setminus U} \abs{\nabla \widetilde{u}_j}^2\,\di x = 0.
	\end{equation*}
	Next, we write
	\begin{equation} \label{eq:case1diff}
		\Lambda - \Lambda_{\emptyset}^C = (\Lambda - \Lambda_2) + (\Lambda_2 - \Lambda_1) + (\Lambda_1 - \Lambda_0) + (\Lambda_0 - \Lambda_{\emptyset}^C),
	\end{equation}
	and estimate each of the differences of the local ND maps. From Propositions~\ref{prop:monoextreme1}--\ref{prop:monoextreme3} we have:
	\begin{align*}
		\inner{f_j,(\Lambda-\Lambda_2)f_j} &\leq -\int_{D_\infty} A_0 \nabla \widetilde{u}_j\cdot\overline{\nabla \widetilde{u}_j}\,\di x \leq -\alpha\int_{B} \abs{\nabla \widetilde{u}_j}^2\,\di x, \\
		\inner{f_j,(\Lambda_2-\Lambda_1)f_j} &\leq \int_{D_0} A_0 \nabla \widetilde{u}_j\cdot\overline{\nabla \widetilde{u}_j}\,\di x \leq \beta\int_{\Omega\setminus U} \abs{\nabla \widetilde{u}_j}^2\,\di x, \\
		\inner{f_j,(\Lambda_1-\Lambda_0)f_j} &\leq \int_{D_\textup{F}} (A_0-A_\textup{F})\nabla \widehat{u}_j\cdot\overline{\nabla \widehat{u}_j}\,\di x \leq \norm{A_0-A_\textup{F}}_*\int_{\Omega\setminus U} \abs{\nabla \widehat{u}_j}^2\,\di x, \\
		\inner{f_j,(\Lambda_0-\Lambda_{\emptyset}^C)f_j} &\leq K\int_{C} \abs{\nabla u_j}^2 \,\di x \leq K\int_{\Omega\setminus U} \abs{\nabla u_j}^2 \,\di x.
	\end{align*}
	Hence, \eqref{eq:case1diff} gives
	\begin{equation*}
		\lim_{j\to\infty} \inner{f_j,\bigl(\Lambda(A_D) - \Lambda_{\emptyset}^C\bigr)f_j} = -\infty,
	\end{equation*}
	so we conclude that $\Lambda(A_D) \not\geq \Lambda_{\emptyset}^C$.
	
	\subsection*{Case 2}
	
	We define two auxiliary coefficients:
	\begin{equation*}
		A_1 = \begin{dcases}
			A_0 & \text{in } \Omega\setminus D_0 \\
			0 I & \text{in } D_0
		\end{dcases} 
		\qquad \text{and} \qquad 
		A_2 = \begin{dcases}
			A_0 & \text{in } \Omega\setminus (D_0 \cup D_\infty) \\
			0 I & \text{in } D_0 \\
			\infty I & \text{in } D_\infty.
		\end{dcases}
	\end{equation*}
	We shorten the notation to $\Lambda = \Lambda(A_D)$ and $\Lambda_k = \Lambda(A_k)$ for $k\in\{0,1,2\}$. From Theorem~\ref{thm:locpot} and Proposition~\ref{prop:locextreme}, we may pick a sequence $(f_j)$ in $L^2_\diamond(\Gamma)$, such that $u_j = u_{f_j}^{A_0}$, $\widehat{u}_j = Eu_{f_j}^{A_1}$, and $\widetilde{u}_j = Eu_{f_j}^{A_2}$ satisfy
	\begin{equation*}
		\lim_{j\to\infty} \int_B \abs{\nabla u_j}^2\,\di x = \lim_{j\to\infty} \int_B \abs{\nabla \widehat{u}_j}^2\,\di x = \lim_{j\to\infty} \int_B \abs{\nabla \widetilde{u}_j}^2\,\di x = \infty
	\end{equation*}
	and
	\begin{equation*}
		\lim_{j\to\infty} \int_{\Omega\setminus U} \abs{\nabla u_j}^2\,\di x = \lim_{j\to\infty} \int_{\Omega\setminus U} \abs{\nabla \widehat{u}_j}^2\,\di x = \lim_{j\to\infty} \int_{\Omega\setminus U} \abs{\nabla \widetilde{u}_j}^2\,\di x = 0.
	\end{equation*}
	Next, we write
	\begin{equation} \label{eq:case2diff}
		\Lambda - \Lambda_{\emptyset}^C = (\Lambda - \Lambda_2) + (\Lambda_2 - \Lambda_1) + (\Lambda_1 - \Lambda_0) + (\Lambda_0 - \Lambda_{\emptyset}^C),
	\end{equation}
	and estimate each of the differences of the local ND maps. Recall that $A_0 - A_\textup{F}$ is negative semidefinite in $U$ and there exists a $c>0$ such that $A_0 - A_\textup{F} \leq -cI$ in $B$. From Propositions~\ref{prop:monoextreme1}--\ref{prop:monoextreme3} and Proposition~\ref{prop:bnd} we have:
	\begin{align*}
		\inner{f_j,(\Lambda-\Lambda_2)f_j} &\leq \int_{D_\textup{F}} A_0 A_\textup{F}^{-1}(A_0-A_\textup{F})\nabla \widetilde{u}_j\cdot\overline{\nabla \widetilde{u}_j}\,\di x \\
		&\leq \norm{A_0 A_\textup{F}^{-1}(A_0-A_\textup{F})}_* \int_{\Omega\setminus U} \abs{\nabla \widetilde{u}_j}^2\,\di x - c(\tfrac{\alpha}{\beta})^2\int_B \abs{\nabla \widetilde{u}_j}^2\,\di x, \\
		\inner{f_j,(\Lambda_2-\Lambda_1)f_j} &\leq -\int_{D_\infty} A_0\nabla \widehat{u}_j\cdot\overline{\nabla \widehat{u}_j}\,\di x \leq 0, \\
		\inner{f_j,(\Lambda_1-\Lambda_0)f_j} &\leq \int_{D_0} A_0\nabla \widehat{u}_j\cdot\overline{\nabla \widehat{u}_j}\,\di x \leq \beta \int_{\Omega\setminus U} \abs{\nabla \widehat{u}_j}^2\,\di x, \\
		\inner{f_j,(\Lambda_0-\Lambda_{\emptyset}^C)f_j} &\leq K\int_{C} \abs{\nabla u_j}^2 \,\di x \leq K\int_{\Omega\setminus U} \abs{\nabla u_j}^2 \,\di x.
	\end{align*}
	Hence, \eqref{eq:case2diff} gives
	\begin{equation*}
		\lim_{j\to\infty} \inner{f_j,\bigl(\Lambda(A_D) - \Lambda_{\emptyset}^C\bigr)f_j} = -\infty,
	\end{equation*}
	so we conclude that $\Lambda(A_D) \not\geq \Lambda_{\emptyset}^C$.
	
	\subsection*{Case 3}
	
	We define two auxiliary coefficients:
	\begin{equation*}
		A_1 = A_{\textup{F}}
		\qquad \text{and} \qquad 
		A_2 = \begin{dcases}
			A_{\textup{F}} & \text{in } \Omega\setminus D_\infty \\
			\infty I & \text{in } D_\infty,
		\end{dcases}
	\end{equation*}
	and denote the coefficient associated with the test-operator $\Lambda_C^{\emptyset}$ as $A_C$.
	
	We shorten the notation to $\Lambda = \Lambda(A_D)$ and $\Lambda_k = \Lambda(A_k)$ for $k\in\{0,1,2\}$. From Theorem~\ref{thm:locpot} and Propositions~\ref{prop:simultloc} and~\ref{prop:locextreme}, we may pick a sequence $(f_j)$ in $L^2_\diamond(\Gamma)$, such that $u_j = u_{f_j}^{A_0}$, $\widehat{u}_j = u_{f_j}^{A_1}$, $\widetilde{u}_j = u_{f_j}^{A_2}$, and $\breve{u}_j = Eu_{f_j}^{A_C}$ satisfy
	\begin{equation*}
		\lim_{j\to\infty} \int_B \abs{\nabla u_j}^2\,\di x = \lim_{j\to\infty} \int_B \abs{\nabla \widehat{u}_j}^2\,\di x = \lim_{j\to\infty} \int_B \abs{\nabla \widetilde{u}_j}^2\,\di x = \lim_{j\to\infty} \int_B \abs{\nabla \breve{u}_j}^2\,\di x = \infty
	\end{equation*}
	and
	\begin{equation*}
		\lim_{j\to\infty} \int_{\Omega\setminus U} \abs{\nabla u_j}^2\,\di x = \lim_{j\to\infty} \int_{\Omega\setminus U} \abs{\nabla \widehat{u}_j}^2\,\di x = \lim_{j\to\infty} \int_{\Omega\setminus U} \abs{\nabla \widetilde{u}_j}^2\,\di x = \int_{\Omega\setminus U} \abs{\nabla \breve{u}_j}^2\,\di x = 0.
	\end{equation*}
	Next, we write
	\begin{equation} \label{eq:case3diff}
		\Lambda_C^{\emptyset} - \Lambda = (\Lambda_C^{\emptyset} - \Lambda_0) + (\Lambda_0 - \Lambda_1) + (\Lambda_1 - \Lambda_2) + (\Lambda_2 - \Lambda),
	\end{equation}
	and estimate each of the differences of the local ND maps. From Propositions~\ref{prop:monoextreme1}--\ref{prop:monoextreme3} we have:
	\begin{align*}
		\inner{f_j,(\Lambda_C^{\emptyset} - \Lambda_0)f_j} &\leq \int_{C} A_0\nabla \breve{u}_j\cdot\overline{\nabla \breve{u}_j}\,\di x \leq \beta \int_{\Omega\setminus U} \abs{\nabla \breve{u}_j}^2\,\di x, \\
		\inner{f_j,(\Lambda_0 - \Lambda_1)f_j} &\leq \int_{D_\textup{F}} (A_\textup{F}-A_0)\nabla u_j\cdot\overline{\nabla u_j}\,\di x \leq \norm{A_\textup{F}-A_0}_* \int_{\Omega\setminus U} \abs{\nabla u_j}^2\,\di x, \\
		\inner{f_j,(\Lambda_1-\Lambda_2)f_j} &\leq K\int_{D_\infty} \abs{\nabla \widehat{u}_j}^2 \,\di x \leq K\int_{\Omega\setminus U} \abs{\nabla \widehat{u}_j}^2 \,\di x, \\
		\inner{f_j,(\Lambda_2-\Lambda)f_j} &\leq -\int_{D_0} A_0\nabla \widetilde{u}_j\cdot\overline{\nabla \widetilde{u}_j}\,\di x \leq -\alpha\int_{B} \abs{\nabla \widetilde{u}_j}^2 \,\di x.
	\end{align*}
	Hence, \eqref{eq:case3diff} gives
	\begin{equation*}
		\lim_{j\to\infty} \inner{f_j,\bigl(\Lambda_C^{\emptyset} - \Lambda(A_D)\bigr)f_j} = -\infty,
	\end{equation*}
	so we conclude that $\Lambda_C^{\emptyset} \not\geq \Lambda(A_D)$.
	
	\subsection*{Case 4}
	
	We define two auxiliary coefficients:
	\begin{equation*}
		A_1 = \begin{dcases}
			A_0 & \text{in } \Omega\setminus D_\infty \\
			\infty I & \text{in } D_\infty
		\end{dcases} 
		\qquad \text{and} \qquad 
		A_2 = \begin{dcases}
			A_0 & \text{in } \Omega\setminus (D_0 \cup D_\infty) \\
			0 I & \text{in } D_0 \\
			\infty I & \text{in } D_\infty,
		\end{dcases}
	\end{equation*}
	and denote the coefficient associated with the test-operator $\Lambda_C^{\emptyset}$ as $A_C$.
	
	We shorten the notation to $\Lambda = \Lambda(A_D)$ and $\Lambda_k = \Lambda(A_k)$ for $k\in\{0,1,2\}$. From Theorem~\ref{thm:locpot} and Proposition~\ref{prop:locextreme}, we may pick a sequence $(f_j)$ in $L^2_\diamond(\Gamma)$, such that $u_j = u_{f_j}^{A_0}$, $\widehat{u}_j = u_{f_j}^{A_1}$, $\widetilde{u}_j = Eu_{f_j}^{A_2}$, and $\breve{u}_j = Eu_{f_j}^{A_C}$ satisfy
	\begin{equation*}
		\lim_{j\to\infty} \int_B \abs{\nabla u_j}^2\,\di x = \lim_{j\to\infty} \int_B \abs{\nabla \widetilde{u}_j}^2\,\di x = \lim_{j\to\infty} \int_B \abs{\nabla \breve{u}_j}^2\,\di x = \infty
	\end{equation*}
	and
	\begin{equation*}
		\lim_{j\to\infty} \int_{\Omega\setminus U} \abs{\nabla u_j}^2\,\di x  = \lim_{j\to\infty} \int_{\Omega\setminus U} \abs{\nabla \widetilde{u}_j}^2\,\di x = \int_{\Omega\setminus U} \abs{\nabla \breve{u}_j}^2\,\di x = 0.
	\end{equation*}
	Next, we write
	\begin{equation} \label{eq:case4diff}
		\Lambda_C^{\emptyset} - \Lambda = (\Lambda_C^{\emptyset} - \Lambda_0) + (\Lambda_0 - \Lambda_1) + (\Lambda_1 - \Lambda_2) + (\Lambda_2 - \Lambda),
	\end{equation}
	and estimate each of the differences of the local ND maps. Recall that $A_\textup{F} - A_0$ is negative semidefinite in $U$ and there exists a $c>0$ such that $A_\textup{F} - A_0 \leq -cI$ in $B$. From Propositions~\ref{prop:monoextreme1}--\ref{prop:monoextreme3} we have:
	\begin{align*}
		\inner{f_j,(\Lambda_C^{\emptyset} - \Lambda_0)f_j} &\leq \int_{C} A_0\nabla \breve{u}_j\cdot\overline{\nabla \breve{u}_j}\,\di x \leq \beta\int_{\Omega\setminus U} \abs{\nabla \breve{u}_j}^2\,\di x, \\
		\inner{f_j,(\Lambda_0 - \Lambda_1)f_j} &\leq K\int_{D_\infty} \abs{\nabla u_j}^2\,\di x \leq K\int_{\Omega\setminus U} \abs{\nabla u_j}^2\,\di x, \\
		\inner{f_j,(\Lambda_1-\Lambda_2)f_j} &\leq -\int_{D_0} A_0\nabla \widehat{u}_j\cdot\overline{\nabla \widehat{u}_j} \,\di x \leq 0, \\
		\inner{f_j,(\Lambda_2-\Lambda)f_j} &\leq \int_{D_\textup{F}} (A_\textup{F}-A_0)\nabla \widetilde{u}_j\cdot\overline{\nabla \widetilde{u}_j}\,\di x \\
		&\leq \norm{A_\textup{F}-A_0}_* \int_{\Omega\setminus U} \abs{\nabla\widetilde{u}_j}^2\,\di x - c\int_B \abs{\nabla\widetilde{u}_j}^2\,\di x.
	\end{align*}
	Hence, \eqref{eq:case4diff} gives
	\begin{equation*}
		\lim_{j\to\infty} \inner{f_j,\bigl(\Lambda_C^{\emptyset} - \Lambda(A_D)\bigr)f_j} = -\infty,
	\end{equation*}
	so we conclude that $\Lambda_C^{\emptyset} \not\geq \Lambda(A_D)$.
	
	\subsection*{Acknowledgements}
	
	The authors are supported by grant 10.46540/3120-00003B from Independent Research Fund Denmark.
	
	\bibliographystyle{plain}

\end{document}